\documentclass[11pt,reqno]{amsart}
\usepackage{xcolor}
 \usepackage{textcomp,multicol,amsmath,amssymb,amsthm,latexsym,dsfont}
 \usepackage{amsfonts,mathtools,wasysym,enumerate, dsfont}
\usepackage[shortlabels]{enumitem}
 
 \usepackage{fourier}
 \usepackage{bbm}
 \usepackage{hyperref}
 \usepackage{color}
 \usepackage{tikz}
 \usepackage{fullpage} 
  \usepackage{comment}
 \allowdisplaybreaks
 
 \usepackage{pgfplots}
\pgfplotsset{compat=1.15}
\usepackage{mathrsfs}
\usetikzlibrary{arrows}

 \newtheorem{theorem}{Theorem}[section]
 \newtheorem{corollary}[theorem]{Corollary}
 \newtheorem{lemma}[theorem]{Lemma}

 \newtheorem{question}[theorem]{Question}
 
   \theoremstyle{definition}

 \newcommand{\CR}{\textrm{\footnotesize CR}}

\title{A note on the width of sparse random graphs}
\author{Tuan Anh Do, Joshua Erde, and Mihyun Kang}
\address{Institute of Discrete Mathematics\\
Graz University of Technology\\
Steyrergasse 30\\
8010 Graz\\
Austria.\\}
\email{\{do,erde,kang\}@math.tugraz.at.}
\begin{document}
\begin{abstract}
In this note, we consider the width of a supercritical random graph according to some commonly studied width measures. We give short, direct proofs of results of Lee, Lee and Oum, and of Perarnau and Serra, on the rank- and tree-width of the random graph $G(n,p)$ when $p= \frac{1+\epsilon}{n}$ for $\epsilon > 0$ constant. Our proofs avoid the use of black box results on the expansion properties of the giant component in this regime, and so as a further benefit we obtain explicit bounds on the dependence of these results on $\epsilon$. Finally, we also consider the width of the random graph in the \emph{weakly supercritical regime}, where $\epsilon = o(1)$ and $\epsilon^3n \to \infty$. In this regime, we determine, up to a constant multiplicative factor, the rank- and tree-width of $G(n,p)$ as a function of $n$ and $\epsilon$.
\end{abstract}
\maketitle
\section{Introduction}
\subsection{Background and motivation}
The \emph{tree-width} of a graph $G$, which we denote by $\textrm{tw}(G)$, is a parameter which broadly measures how similar the global structure of $G$ is to that of a tree, via the existence of a \emph{tree-decomposition}, which displays the structure of $G$ in a tree-like fashion. Tree-decompositions were originally studied by Halin \cite{Halin}, and later independently by Robertson and Seymour as an integral tool in their proof of Wagner's conjecture \cite{RobertsonXX}, and have since then been the subject of much study.

Beyond their usefulness as a tool in structural graph theory, tree-decompositions have also turned out to be important in the field of computational complexity. Many problems which are computationally hard on general graphs can be solved in polynomial time on graphs of bounded tree-width (see Bodlaender \cite{Bodlaender}). As the subject developed, many other notions of `width' for graphs, and other combinatorial structures, have been considered, each with their own applications to structural problems, as well as computational problems. One particular example that we will consider in this note is the \emph{rank-width} of a graph $G$, which we denote by $\textrm{rw}(G)$. Roughly, if low tree-width implies that a graph has a `tree-like' decomposition over small vertex cuts, where small here is measured in terms of their cardinality,  low rank-width implies that the graph has a `tree-like' decomposition over \emph{simple} edge cuts, where the simplicity of these edge cuts is not measured in terms of just the number of edges crossing them, but rather in terms of an algebraic measure more akin to the complexity of the edge cut. In particular, both very sparse and very dense edge cuts are simple in this way, and so, unlike the case of tree-width, there can be quite dense graphs with low rank-width. In this way, the graphs of low rank-width are a broader class of `low-complexity' graphs than the graphs of low tree-width, and in particular, Oum \cite{Oum} showed that the rank-width of a graph is always at most one more than the tree-width. We will give precise definitions of the relevant terms later in Section \ref{s:Pre}; for more background on tree-width and other width parameters see the survey of Hlin{\v{e}}n{\`y}, Oum, Seese and Gottlob \cite{Hlineny}.

In this paper, we are interested in the width of \emph{random graphs}. The \emph{binomial random graph model} $G(n,p)$, introduced by Gilbert \cite{Gilbert}, is a random variable distributed on the subgraphs of the complete graph $K_n$ where we retain each edge independently with probability $p$. The tree-width of $G(n,p)$ was first considered by Kloks \cite{Kloks}, who showed that if $p = \frac{d}{n}$ for $d \geq 2.36$, then with probability tending to one as $n \to \infty$ (whp) the tree-width of $G(n,p)$ is at least $r n$ for some $r:= r(d)>0$. Furthermore, he showed that one can take $r(d) \rightarrow 1$ as $d \rightarrow \infty$, which can be seen to be optimal, as the tree-width of a graph on $n$ vertices cannot be larger than $n$. Later, Gao \cite{Gao} improved the lower bound on $d$ to $2.16$, and asked if this could be further improved to $d \geq 1$.

This is a natural question to ask, due to the \emph{phase transition} which $G(n,p)$ undergoes at the critical point $p=\frac{1}{n}$. More explicitly, work of Erd\H{o}s and R\'{e}nyi \cite{E-R} implies that for $p= \frac{d}{n}$ with $d < 1$, whp $G(n,p)$ only has \emph{small} components, of logarithmic order, each of which is a tree or unicyclic. Since all trees and unicyclic graphs have tree-width at most two, it follows that whp in this regime $G(n,p)$ has tree-width at most two as well. However, when $p= \frac{d}{n}$ with $d > 1$, whp $G(n,p)$ contains a unique \emph{giant} component, whose order is linear in $n$ and which is known to have quite a complex structure; see \cite{BollobasBook,Janson,Frieze} for a more detailed introduction to random graphs.

This question of Gao \cite{Gao} was finally answered by Lee, Lee and Oum \cite{LeeLeeOum}, who instead considered the rank-width of the random graph $G(n,p)$ and gave the following, almost complete, description of how this parameter behaves for different ranges of $p$.

\begin{theorem}[{\cite[Theorem 1.1]{LeeLeeOum}}]\label{t:rankwidthrandom}
For a random graph $G:=G(n,p)$, whp the following statements hold:
\begin{enumerate}[(i)]
    \item If $p \in (0,1)$ is constant, then $\textrm{rw}(G) = \lceil \frac{n}{3} \rceil - O(1)$;
    \item If $\frac{1}{n} \ll p \leq \frac{1}{2}$, then $\textrm{rw}(G) = \lceil \frac{n}{3} \rceil - o(n)$;
    \item\label{i:superrank} If $p = \frac{d}{n}$ and $d> 1$, then  $\textrm{rw}(G) \geq r n$ for some $r := r(d)>0$;
    \item If $p=\frac{1}{n}$, then whp $\textrm{rw}(G) = O\left(n^{\frac{2}{3}}\right)$;
    \item If $p \leq \frac{d}{n}$ and $d <1$, then $\textrm{rw}(G) \leq 2$.
\end{enumerate}
\end{theorem}

Note that, since the rank-width of an $n$-vertex graph can be shown to be at most $\frac{n}{3}$, the first result is tight up to a lower-order additive term for larger $p$. Furthermore, by the aforementioned result of Oum \cite{Oum} that for any graph $G$
\begin{equation}\label{e:twrw}
\textrm{rw}(G) \leq \textrm{tw}(G) +1,
\end{equation}
Theorem \ref{t:rankwidthrandom} \ref{i:superrank} implies that the tree-width of $G\left(n,\frac{d}{n}\right)$ is linear in $n$ for any $d>1$, giving a positive answer to the question of Gao \cite{Gao}.

Later, Perarnau and Serra \cite{Serra} gave a more direct proof of this, and also considered more carefully the tree-width of the critical random graph.
\begin{theorem}[\cite{Serra}]\label{t:treewidthrandom}
For a random graph $G:=G(n,p)$ whp the following hold:
\begin{enumerate}[(i)]
    \item\label{i:supertree} If $p = \frac{d}{n}$ and $d> 1$, then  $\textrm{tw}(G) \geq r' n$ for some $r' := r'(d)>0$;
    \item\label{i:critical} If $p = \frac{1}{n}$, then $\textrm{tw}(G) = O(1)$.
\end{enumerate}
\end{theorem}
We note that their proof of Theorem \ref{t:treewidthrandom} \ref{i:critical} in fact holds for any $p$ inside the critical window.

From Theorems \ref{t:rankwidthrandom} and \ref{t:treewidthrandom}, together with \eqref{e:twrw}, it follows that both the rank- and tree-width of $G(n,p)$ are bounded above by a constant when $p$ is below or inside the critical window, i.e., when there exists a constant $\lambda > 0$ such that $p \leq \frac{1+ \lambda n^{-\frac{1}{3}}}{n}$, but that both the rank- and tree-width are linearly large in the supercritical regime. This raises the natural question of whether this transition happens \emph{smoothly}.

Both Theorem \ref{t:rankwidthrandom} and Theorem \ref{t:treewidthrandom} rely on a deep result of Benjamini, Kozma and Womald \cite{Bejamini2014} on the \emph{expansion} properties of the giant component of $G(n,p)$ in the supercritical regime. A graph is an expander if it satisfies a type of discrete isoperimetric inequality. Notions of graph expansion have turned out to be of fundamental importance to various topics in combinatorics and computer science. For a comprehensive introduction to expander graphs, see the survey of Hoory, Linial and Widgerson \cite{HLW06}. In particular, graph expansion has been used as a tool in the study of random structures, see, for example, the survey paper of Krivelevich \cite{K19}.

In \cite[Section 4]{LeeLeeOum} it is explained how the results in \cite{Bejamini2014} imply the following theorem. Given a set of vertices $U \subseteq V(H)$ in a graph $H$, let us write $d(U) = \sum_{v \in U} d(v)$. The \emph{Cheeger constant} of a graph $H$ is \[
\Phi (H) := \min_{S \subset  V(H)} \frac{e(S,S^c)}{\min \left\{ d(S), d(S^c) \right\} }
\]
where $S^c = V(H) \setminus S$. A graph $H$ with $\Phi(H) \geq \alpha$ is an \emph{$\alpha$-edge-expander}. 
\begin{theorem}[\cite{Bejamini2014}]\label{t:expansion}
Let $d >1$ and let $p=\frac{d}{n}$. Then there exist $\alpha,\delta >0$ such that whp $G(n,p)$ contains a connected subgraph $H$ with $\Phi(H) \geq \alpha$ and $|V(H)| \geq \delta n$.
\end{theorem}
Perarnau and Serra \cite{Serra} note that the proof of Theorem \ref{t:rankwidthrandom} \ref{i:superrank} in \cite{LeeLeeOum} shows that 
 \begin{equation}\label{e:LLO}
r(d) \geq \frac{\alpha \delta}{L}
\end{equation}
where $L$ is some implicit constant, which can be shown to grow like $\Omega\left(\log \frac{1}{d-1} \right)$ as $d \rightarrow 1$. On the other hand they note that their own proof of Theorem \ref{t:treewidthrandom} \ref{i:supertree} gives
\begin{equation}\label{e:PS}
r'(d) \geq \frac{(\alpha \delta)^2}{9e^3d^2}.
\end{equation}
However, it is not clear how the constants $\alpha$ and $\delta$ in Theorem \ref{t:expansion} behave as a function of $d$ as $d \rightarrow 1$, so it is difficult to compare these two bounds.

An alternative, and more elementary, proof of Theorem \ref{t:expansion} can be derived from the work of Krivelevich \cite{Krivelevich}, who shows the existence of a linear sized bounded degree vertex-expander in this range of $p$. 
The dependence of the result on $\alpha$ and $\delta$ is not made explicit, but a careful reading of the proof leads to the following bounds:
\[
\alpha = \exp \left( - \Omega\left((d-1)^{-2}\right)\right) \qquad \text{ and } \qquad \delta = \left( \frac{(d-1)^3}{\left(\log \frac{1}{d-1}\right)^2}\right).
\]
However, we should note that there does not seem to be any attempt in \cite{Krivelevich} to optimise the arguments with respect to these parameters, and it seems unlikely that the above bounds are optimal. In particular, known results about the structure of the giant component in this regime \cite{DLP14} suggest that the real growth rates should be closer to $\alpha=\Omega\left(d-1\right)$ and $\delta = \Omega\left((d-1)^2\right)$.

\subsection{Main results} 
Using a technique of Luczak and McDiarmid \cite{LuczakMcdiarmid}, we are able to give short direct proofs of these results which avoid using Theorem \ref{t:expansion}, and in particular give an explicit lower bound on the rank- and tree-width of $G(n,p)$ for $p=\frac{d}{n}$ for any $d > 1$. Note, in the following theorems, for ease of presentation, we parameterize our results in terms of $\epsilon = d-1$.

\begin{theorem}\label{t:rankwidthdirect}
Let $\epsilon >0$ be sufficiently small and let $p=\frac{1+\epsilon}{n}$. Then whp 
$$\textrm{rw}\left(G\left(n,p\right)\right) = \Omega\left( \frac{\epsilon^3}{\left(\log \frac{1}{\epsilon}\right)^3} \right) n.$$
\end{theorem}

From this it is possible to give a corresponding lower bound on the tree-width of $G(n,p)$, however with a more direct argument we can remove some of the polylogarithmic factors from the result.

\begin{theorem}\label{t:treewidthdirect}
Let $\epsilon>0$ be sufficiently small and let $p=\frac{1+\epsilon}{n}$. Then whp 
$$\textrm{tw}\left(G\left(n,p\right)\right) = \Omega\left( \frac{\epsilon^3}{\log \frac{1}{\epsilon}} \right) n.$$
\end{theorem}

Since we are interested in the behaviour of these parameters as functions of $\epsilon$ as $\epsilon \to 0$, we have stated the above results only for the \emph{barely supercritical regime}. However, we note that, since tree-width is an increasing graph parameter, Theorem \ref{t:treewidthdirect} clearly implies Theorem \ref{t:treewidthrandom} \ref{i:supertree}. In the case of Theorem \ref{t:rankwidthdirect} it is not immediate, but it is a simple exercise to adapt the proof to cover the range of $p$ in Theorem \ref{t:rankwidthrandom} \ref{i:superrank}. In fact, it can be seen that both of these proofs are even effective all the way to the critical window, holding in the \emph{weakly supercritical regime} where $\epsilon^3 n \to \infty$.

However, neither Theorem \ref{t:rankwidthdirect} nor \ref{t:treewidthdirect} is optimal in terms of their dependence on $\epsilon$ when $\epsilon=o(1)$. Nevertheless, using a different method, also based on the expansion properties of $G(n,p)$, we can remove the extra polylogarithmic terms in Theorems \ref{t:rankwidthdirect} and \ref{t:treewidthdirect}, and obtain an asymptotically optimal bound on the rank- and tree-width of $G(n,p)$ in this regime.

\begin{theorem}\label{t:weaklysup}
Let $\epsilon=\epsilon(n) >0$ be such that $\epsilon=o(1)$ and $\epsilon^3 n \to \infty$, and let $p = \frac{1+\epsilon}{n}$. Then whp
\[
\textrm{rw}\left(G\left(n,p\right)\right) = \Theta\left(\epsilon^3 n \right)  \qquad \text{ and } \qquad \textrm{tw}\left(G\left(n,p\right)\right) = \Theta\left(\epsilon^3 n \right) .\]
\end{theorem}

We note that the upper bounds in Theorem \ref{t:weaklysup} follow quite easily from relatively standard facts about $G(n,p)$. Indeed, it is well-known that in this regime of $p$ whp the largest component $L_1$ of $G(n,p)$ has $\Theta(\epsilon n)$ vertices and $\Theta(\epsilon^3 n)$ \emph{excess} edges, and all other components are trees or unicyclic. In particular, the tree-width of all components but the largest is at most one, and so the tree-width of $G(n,p)$ is at most $\max\{ 1, \textrm{tw}(L_1) \}$. The bound then follows from the fact that tree-width of a graph is at most the number of excess edges. 

\subsection{Techniques and outline of the paper}
A well-known result in structural graph theory says that the tree-width of a graph $G$ can be bounded from below by the size of the smallest \emph{balanced separation}, a partition of $V(G)$ into three pieces $(A,S,B)$ with $\min |A|,|B| \geq \frac{|V(G)\setminus S|}{3}$ and $e(A,B)=0$, i.e., there are no edges between $A$ and $B$. The \emph{size} of such a separation is the order $|S|$ of the separator. The idea of the proofs of Kloks \cite{Kloks} and of Gao \cite{Gao} is to bound the tree-width by showing the likely non-existence of small separators using a union bound, which turns out to only be effective when $\epsilon = np-1$ is sufficiently large, due to the large number of possible separators.

A useful trick that can reduce the possible number of separators to consider, following a technique of Luczak and McDiarmid \cite{LuczakMcdiarmid}, is to first find a spanning tree of the graph, and since tree-width is decreasing under taking subgraphs, it suffices to first find a large tree in the random graph. Now, classical results on the phase transition already imply that whp a supercritical random graph contains a tree $T$ on $\Omega(\epsilon n)$ vertices, and with a bit of care it is possible to control the maximum degree of this tree (as a function of $\epsilon$).

Restricting our attention to the vertex set $V(T)$, we can show with a sprinkling argument that it is very unlikely that any fixed partition of $V(T)$ forms a balanced separation. Unfortunately, since each subset $S \subseteq V(T)$ could potentially be the separator in many balanced separations, we cannot naively complete the argument with a union bound.

However, since $T$ spans $V(T)$, any subset $S\subseteq V(T)$ splits $T$ into at most $\Delta(T) |S|$ many components, and hence there are at most $2^{\Delta(T) |S|}$ separations whose separator is $S$. This rather simple observation turns out to be powerful enough to use a union bound to show the likely non-existence of small balanced separators.

We note that, in a similar manner, the existence of such a tree reduces the possible number of small balanced edge cuts of $V(T)$, an idea already used by Spencer and T\'{o}th \cite{SpencerToth} in their work on the crossing number of random graphs. This in turn, in a similar manner as for tree-width, can be used to give a lower bound for the rank-width of $G(n,p)[V(T)]$, allowing us to use this technique to also prove Theorem \ref{t:rankwidthdirect}. 

For the proof of Theorem \ref{t:weaklysup}, we follow a similar approach to that of Lee, Lee and Oum \cite{LeeLeeOum}, and of Perarnau and Serra \cite{Serra}, and consider the expansion properties of the giant component of $G(n,p)$. However, unlike in the supercritical regime, it is not the case that the giant component typically contains a large expanding subgraph in this regime of $p$. Instead, using standard estimates on the degree distribution of the kernel of $G(n,p)$, we show that whp the kernel contains as an \emph{induced topological minor} a large subgraph $H$ which is distributed as a random $3$-regular graph, which is known to whp have good expansion properties. This implies that $H$ does not contain any small balanced edge cuts, which, together with the fact that $\Delta(H) =3$, is sufficient to bound the rank-width of $H$ from below. Finally, since rank-width can be shown to be decreasing under taking induced topological minors, we can conclude that the rank-width of $G(n,p)$ is at least the rank-width of $H$.

The paper is structured as follows. In Section \ref{s:Pre} we give preliminary definitions and results. We give our proofs of Theorems \ref{t:rankwidthdirect} and \ref{t:treewidthdirect} in Section \ref{s:direct} and then prove Theorem \ref{t:weaklysup} in Section \ref{s:weakly}. Some final discussion of open and related problems is given in Section \ref{s:discussion}.

\subsection*{Notation}
Unless otherwise specified, all logarithms will be the natural logarithm. Throughout the paper we will suppress floor and ceiling signs for ease of presentation.

\section{Preliminaries}\label{s:Pre}
A \textit{tree-decomposition} of a graph $G$ is a pair $(T, \mathcal{V})$ where $T$ is a tree and $\mathcal{V} = \{ V_x \subseteq V(G) \colon x \in V(T) \}$ is a family of subsets of $V(G)$,
such that the following conditions hold:
\begin{itemize}
    \item For all $v\in V(G)$, the set $\left\{x\in V(T): v\in V_x\right\}$ induces a non-empty subtree of $T$;
    \item For all $uv\in E(G) $, there is some bag $V_x$ containing both $u$ and $v$.
\end{itemize}
The \textit{width} of $(T, \mathcal{V})$ is $\max \{ |V_x| - 1 \colon x \in V(T) \}$. The \textit{tree-width} $\text{tw}(G)$ of $G$ is the minimum width over all tree-decompositions of $G$. 

We say that a set $S\subseteq V(G)$ is a \emph{$(k,\alpha)$-separator} if $|S| = k$ and every component of $G \setminus S$ has size at most $\alpha|V(G) \setminus S|$. The \emph{$\alpha$-separation number} $\text{sep}_{\alpha}(G)$ of $G$ is the smallest size $k$ of a $(k,\alpha)$-separation in $G$. A classic result of Robertson and Seymour \cite{Robertson} bounds the tree-width of a graph from below by the size of the smallest $\left(k,\frac{1}{2}\right)$-separation, although it is phrased in different terms. 
\begin{lemma}[See \cite{Harvey}]\label{l:twsep}
For any graph $G$, $\textrm{tw}(G) \geq \text{sep}_{\frac{1}{2}}(G) -1$.
\end{lemma}

For ease of presentation we will want to work with a slightly different notion of a balanced separation. We call a partition $(A,S,B)$ of $V(G)$ a \emph{ $\left(k,\frac{1}{2}\right)$-balanced partition} if $|S| \leq k$, $e(A,B)=0$, and $\frac{1}{3} |V(G) \setminus S| \leq |A|,|B| \leq \frac{2}{3} |V(G) \setminus S|$.

\begin{lemma}\label{l:balancedsep}
Let $G=(V,E)$ be a graph and let $S$ be a $\left(k,\frac{1}{2}\right)$-separator of $G$. Then there exist $A,B \subseteq V$ (not necessarily unique) such that $(A,S,B)$ is a $\left(k,\frac{1}{2}\right)$-balanced partition.
\end{lemma}
\begin{proof}
Let $G'=G[V \setminus S]$ and let  $A_1,A_2, \ldots, A_s$ be the vertex sets of the components of $G'$, in non-increasing order of size. Let $j$ be minimal such that $\left| \bigcup_{i=1}^j A_i \right| \geq \frac{1}{3}|G'|$. Since $S$ is a $\left(k,\frac{1}{2}\right)$-separator, $|A_1| \leq \frac{1}{2} |G'| \leq \frac{2}{3} |G'|$, and if $j \neq 1$ then $|A_j| \leq |A_1| \leq \frac{1}{3} |G'|$ and so $\left| \bigcup_{i=1}^j A_i \right| \leq \frac{2}{3} |G'|$. In either case it follows that $A =  \bigcup_{i=1}^j A_i$ and $B =  \bigcup_{i=j+1}^s A_i$ is the desired partition.
\end{proof}

Given a graph $G$ and two subsets $V_1,V_2 \subseteq V(G)$, we let $N_{V_1,V_2}$ be the adjacency matrix whose rows are labelled by $V_1$ and whose columns are labelled by $V_2$, so that the entry $\left(N_{V_1,V_2}\right)_{v_1,v_2} = 1$ if and only if $v \in V_1$ and $v_2 \in V_2$ are adjacent, and otherwise it is $0$. The \emph{cutrank} of $V_1$ and $V_2$, which we denote by $\rho_G(V_1,V_2)$, is the rank of $N_{V_1,V_2}$ over $\mathbb{F}_2$, which we write as $\textrm{rank}(N_{V_1,V_2})$. A tree is \emph{subcubic} if every vertex has degree one or three. A \emph{rank-decomposition} of a graph $G$ is a pair $(T,f)$ where $T$ is a subcubic tree and $f$ is a bijection from $V(G)$ to the set of leaves of $T$.

Deleting an edge $e=uv$ of $T$ splits $T$ into two components $C_u$ and $C_v$ containing $u$ and $v$, respectively. If we let $A_{uv} = f^{-1}(C_u)$ and $B_{uv} = f^{-1}(C_v)$, then we can define the \emph{rank-width} of $G$ to be
\[
\textrm{rw}(G) = \min_{(T,f)} \max_{uv \in E(T)} \rho_G(A_{uv},B_{uv}),
\]
where we take the minimum over all rank-decompositions $(T,f)$ of $G$.
We will use the following lemmas of Lee, Lee and Oum \cite{LeeLeeOum}. The first is a relatively standard statement bounding a width parameter from below by the non-existence of a balanced separation of low order.

\begin{lemma}[{\cite[Lemma 2.1]{LeeLeeOum}}]\label{l:smallcut}
Let $G$ be a graph with at least two vertices. If the rank-width of $G$ is at most $k$, then there exists a partition $V(G) = A \cup B$ such that $|A|,|B| \geq  \frac{|G|}{3}$ and $\rho_G(V_1,V_2) \leq k$.
\end{lemma}

The second is a useful technical lemma which bounds the rank of a matrix in terms of the size of its support.

\begin{lemma}[{\cite[Lemma 4.3]{LeeLeeOum}}]\label{l:rank}
Let $A$ be a matrix over $\mathbb{F}_2$ with at least $n$ non-zero entries. If each row and column contains at most $M$ non-zero entries then $\textrm{rank}(A) \geq \frac{n}{M^2}$.
\end{lemma}

The \emph{bisection width} $b(G)$ of a graph $G$ is the minimum of $e(A,B)$ over all partitions $(A,B)$ of $V(G)$ such that $\frac{1}{3}|G|\leq |A|, |B|\leq \frac{2}{3}|G|$. We note that, Lemmas \ref{l:smallcut} and \ref{l:rank} allow us to bound the rank-width of a graph in terms of its bisection width and maximum degree.

Also, the following bound relating the rank- and tree-width of a graph will be useful.

\begin{theorem}[{\cite[Theorem 3]{Oum}}]\label{t:rwtw}
For every graph $G$
\[
\textrm{rw}(G) \leq \textrm{tw}(G) +1.
\]

\end{theorem}

If $G$ is a graph and $v \in V(G)$, then the \emph{local complementation of $G$ at $v$}, denoted by $G*v$, is the graph whose vertex set is $V(G)$ and whose edge set is the same as $E(G)$, except adjacency and non-adjacency are reversed in $N(v)$, see Figure \ref{f:complementation}. Two graphs are \emph{locally equivalent} if one can be obtained from the other by a sequence of local complementations. A graph $H$ is a \emph{vertex-minor} of $G$ if $H$ is an induced subgraph of a graph which is locally equivalent to $G$.

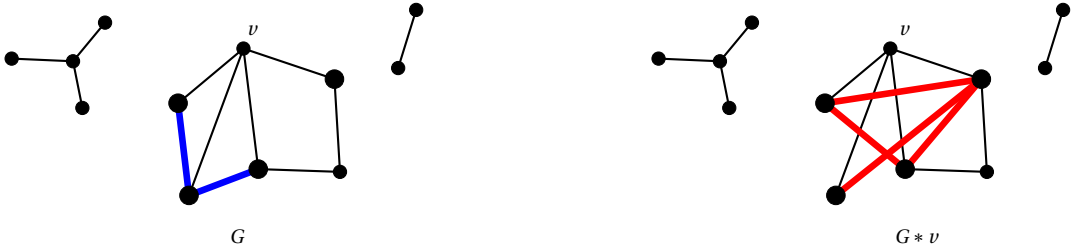
\begin{figure}[!ht]
    \centering
\definecolor{ududff}{rgb}{0.30196078431372547,0.30196078431372547,1.}
\definecolor{ffqqqq}{rgb}{1.,0.,0.}
\definecolor{qqqqff}{rgb}{0.,0.,1.}
\begin{tikzpicture}[line cap=round,line join=round,>=triangle 45,x=1.0cm,y=1.0cm]
\clip(3.,0.6) rectangle (18.,4.2);
\draw [line width=.8pt] (4.144345898004438,3.1845676274944568)-- (4.570066518847009,3.6989800443458978);
\draw [line width=.8pt] (4.144345898004438,3.1845676274944568)-- (3.3283813747228415,3.220044345898004);
\draw [line width=.8pt] (4.144345898004438,3.1845676274944568)-- (4.268514412416854,2.5637250554323727);
\draw [line width=.8pt] (6.4099778270510015,3.3539689578713974)-- (5.540798226164083,2.62669623059867);
\draw [line width=.8pt] (6.4099778270510015,3.3539689578713974)-- (6.606760681217132,1.7499862910821482);
\draw [line width=2.5pt,color=qqqqff] (5.540798226164083,2.62669623059867)-- (5.685959575664167,1.4000818709720222);
\draw [line width=2.5pt,color=qqqqff] (5.685959575664167,1.4000818709720222)-- (6.606760681217132,1.7499862910821482);
\draw [line width=.8pt] (8.706187201877892,3.867828833853963)-- (8.466778914434121,3.0943559051894742);
\draw [line width=.8pt] (6.4099778270510015,3.3539689578713974)-- (5.685959575664167,1.4000818709720222);
\draw [line width=.8pt] (7.69330598576963,1.7131542468600296)-- (6.606760681217132,1.7499862910821482);
\draw [line width=.8pt] (7.619641897325393,2.947027728301)-- (7.69330598576963,1.7131542468600296);
\draw [line width=.8pt] (6.4099778270510015,3.3539689578713974)-- (7.619641897325393,2.947027728301);
\draw [line width=.8pt] (12.744345898004408,3.1845676274944568)-- (13.170066518846978,3.6989800443458978);
\draw [line width=.8pt] (12.744345898004408,3.1845676274944568)-- (11.928381374722814,3.220044345898004);
\draw [line width=.8pt] (12.744345898004408,3.1845676274944568)-- (12.868514412416824,2.5637250554323727);
\draw [line width=.8pt] (15.009977827050971,3.3539689578713974)-- (14.140798226164053,2.62669623059867);
\draw [line width=.8pt] (15.009977827050971,3.3539689578713974)-- (15.206760681217101,1.7499862910821482);
\draw [line width=.8pt] (17.306187201877886,3.867828833853963)-- (17.06677891443411,3.0943559051894742);
\draw [line width=.8pt] (15.009977827050971,3.3539689578713974)-- (14.285959575664137,1.4000818709720222);
\draw [line width=.8pt] (16.293305985769603,1.7131542468600296)-- (15.206760681217101,1.7499862910821482);
\draw [line width=.8pt] (16.219641897325367,2.947027728301)-- (16.293305985769603,1.7131542468600296);
\draw [line width=.8pt] (15.009977827050971,3.3539689578713974)-- (16.219641897325367,2.947027728301);
\draw [line width=2.5pt,color=ffqqqq] (14.140798226164053,2.62669623059867)-- (16.219641897325367,2.947027728301);
\draw [line width=2.5pt,color=ffqqqq] (15.206760681217101,1.7499862910821482)-- (16.219641897325367,2.947027728301);
\draw [line width=2.5pt,color=ffqqqq] (14.285959575664137,1.4000818709720222)-- (16.219641897325367,2.947027728301);
\draw [line width=2.5pt,color=ffqqqq] (15.206760681217101,1.7499862910821482)-- (14.140798226164053,2.62669623059867);
\begin{scriptsize}
\draw [fill=black] (4.144345898004438,3.1845676274944568) circle (2.5pt);
\draw [fill=black] (4.570066518847009,3.6989800443458978) circle (2.5pt);
\draw [fill=black] (3.3283813747228415,3.220044345898004) circle (2.5pt);
\draw [fill=black] (4.268514412416854,2.5637250554323727) circle (2.5pt);
\draw [fill=black] (6.4099778270510015,3.3539689578713974) circle (2.5pt);
\draw [fill=black] (5.540798226164083,2.62669623059867) circle (3.5pt);
\draw [fill=black] (6.606760681217132,1.7499862910821482) circle (3.5pt);
\draw [fill=black] (5.685959575664167,1.4000818709720222) circle (3.5pt);
\draw [fill=black] (8.706187201877892,3.867828833853963) circle (2.5pt);
\draw [fill=black] (8.466778914434121,3.0943559051894742) circle (2.5pt);
\draw [fill=black] (7.619641897325393,2.947027728301) circle (3.5pt);
\draw [fill=black] (7.69330598576963,1.7131542468600296) circle (2.5pt);
\draw [fill=black] (12.744345898004408,3.1845676274944568) circle (2.5pt);
\draw [fill=black] (13.170066518846978,3.6989800443458978) circle (2.5pt);
\draw [fill=black] (11.928381374722814,3.220044345898004) circle (2.5pt);
\draw [fill=black] (12.868514412416824,2.5637250554323727) circle (2.5pt);
\draw [fill=black] (15.009977827050971,3.3539689578713974) circle (2.5pt);
\draw [fill=black] (14.140798226164053,2.62669623059867) circle (3.5pt);
\draw [fill=black] (15.206760681217101,1.7499862910821482) circle (3.5pt);
\draw [fill=black] (14.285959575664137,1.4000818709720222) circle (3.5pt);
\draw [fill=black] (17.306187201877886,3.867828833853963) circle (2.5pt);
\draw [fill=black] (17.06677891443411,3.0943559051894742) circle (2.5pt);
\draw [fill=black] (16.219641897325367,2.947027728301) circle (3.5pt);
\draw [fill=black] (16.293305985769603,1.7131542468600296) circle (2.5pt);
\draw[color=black] (6.533096592772894,3.6) node {$v$};
\draw[color=black] (15.2,3.6) node {$v$};
\draw (6.33, 0.85) node{$G$};
\draw  (15.37, 0.85) node {$G*v$};
\end{scriptsize}
\end{tikzpicture}
    \caption{The local complementation $G*v$ of a graph $G$ at a vertex $v$.}
    \label{f:complementation}
\end{figure}

It is shown in Oum \cite{Oum2005} that two locally equivalent graphs have the same rank-width, and so it is easy to see that 
\begin{equation}\label{e:vtxminor}
\text{if $H$ is a vertex-minor of $G$ then $\textrm{rw}(H) \leq \textrm{rw}(G)$,}
\end{equation}
since rank-width is non-increasing when taking induced subgraphs. A simple consequence of this fact and the following lemma is that rank-width is also non-increasing when taking induced topological minors.

\begin{lemma}\label{l:indtopminorvtxminor}
If $H$ is an induced topological minor of $G$ then $H$ is a vertex-minor of $G$.
\end{lemma}
\begin{proof}
It is clearly sufficient to prove the lemma in the case where $G$ is obtained from $H$ by subdividing a single edge $e = uv$ by a new vertex $x$. However, since $N_G(x) = \{u,v\}$ and $uv \not\in E(G)$, it follows that $E(G*x) = E(G) \cup  \{uv\}$ and so $H = G*x[V(H)]$. In particular, $H$ is a vertex-minor of $G$.
\end{proof}

We will use the following lemma of Krivelevich \cite{Krivelevich} on high degree vertices in $G(n,p)$.

\begin{lemma}[{\cite[Proposition 2]{Krivelevich}}]\label{l:highdegree}
Let $d \geq 0$, let $p = \frac{d}{n}$ and let $\delta >0$ be sufficiently small. Then whp every set of $\frac{\delta}{\log \frac{1}{\delta}} n$ vertices in $G(n,p)$ touches at most $\delta n$ edges.
\end{lemma}

We will need the following simple bound on the expectation of a truncated binomial distribution.

\begin{lemma}[{\cite[Lemma 2.5]{Erde}}]\label{l:restricted}
		Let $X \sim \text{Bin}(n,p)$ be a binomial random variable with $2enp < K$ for some constant $K>0$. If $Y = \min\{ X , K\}$, then
		\[
		\mathbb{E}(Y) \geq np - K2^{-K}.
		\]
	\end{lemma}

We will use the following Chernoff-type bound on the tail probabilities of the binomial distribution, see e.g., \cite[Appendix A]{Alon}.
\begin{lemma}\label{l:chernoff}
Let $n \in \mathbb{N}$, let $p \in [0,1]$, and let $X \sim \text{Bin}(n,p)$. Then for every positive $a$ with $a \leq \frac{np}{2}$,
\[
\mathbb{P}\left(\left|X -np \right| > a\right) < 2 \exp\left(-\frac{a^2}{4np} \right).
\]
\end{lemma}

We will also need the following generalised Chernoff-type bound, due to Hoeffding.
	
	\begin{lemma}[\cite{H63}]\label{l:hoeffding}
		Let $K>0$ be a constant and let $X_1,\ldots, X_n$ be independent random variables such that $0\leq X_i \leq K$ for each $1\leq i \leq n$. If $X = \sum_{i=1}^n X_i$ and $t\geq 0$ then
		\[
		\mathbb{P}\left(|X-\mathbb{E}(X)| \geq t\right) \leq 2 \exp\left(-\frac{t^2}{nK^2}\right).
		\]
	\end{lemma}

In Section \ref{s:weakly} we will use the \emph{configuration model}. Given a degree sequence ${\bf d} \in \mathbb{N}^s$, the configuration model constructs a multigraph $G^*({\bf d})$ as follows: We start with a set of \emph{cells} $\mathcal{W}({\bf d})=\{W_1,\ldots, W_s\}$ where $(|W_1|,|W_2|,\ldots, |W_s|) = {\bf d}$. We call the points in the $W_i$ \emph{half-edges} and we say the \emph{degree} of a cell $W_i$ is $|W_i|$. A \emph{configuration} is a partition $M$ of $W := \bigcup_{i \in [s]} W_i$ into pairs, which we think of as a perfect matching on the set of half-edges. The graph $G^*({\bf d})$ is formed by choosing a configuration $M$ uniformly at random and taking the graph $G(\mathcal{W},M)$ whose vertex set is $[s]$, the number of edges between $i \neq j$ is the number of half-edges in $W_i$ which are matched to a half-edge in $W_j$, and the number of loops at $i$ is half the number of half-edges in $W_i$ which are matched to other half-edges in $W_i$. If ${\bf{d}}=(d,d,\ldots,d) \in \mathbb{N}^m$, then $G^*(\bf{d})$ is a random $d$-regular multigraph on $m$ vertices, which we will denote by $G^*(m,d)$. For more details on the configuration model, see for example \cite{Frieze}.

\section{General bounds: Proofs of Theorems \ref{t:rankwidthdirect} and \ref{t:treewidthdirect}}\label{s:direct}
We start with the proof of Theorem \ref{t:rankwidthdirect}, which is slightly simpler. We note that the proof essentially has two ingredients: The first is using the technique of Luczak and McDiarmid \cite{LuczakMcdiarmid} to bound from below the bisection width of a linear sized subgraph of $G(n,p)$, which already appears in the paper of Spencer and T\'{o}th \cite{SpencerToth}. The second is to use Lemmas \ref{l:smallcut} and \ref{l:rank} to turn this into a bound on the rank-width of $G(n,p)$. However, in order to do so we first need to delete a small number of edges to reduce the maximum degree. This part is similar to the argument of Lee, Lee and Oum \cite{LeeLeeOum}, who instead use the result of Benjamini, Kozma and Wormald \cite{Bejamini2014} to bound from below the bisection width of an appropriate subgraph.

\begin{proof}[Proof of Theorem \ref{t:rankwidthdirect}]
We argue via a sprinkling argument, generating two random graphs $G(n,p_1)$ and $G(n,p_2)$ independently such that $(1-p_1)(1-p_2) = 1-p$ so that their union $G(n,p_1) \cup G(n,p_2)$ has the same distribution as $G(n,p)$. Explicitly we take $p_1 = \frac{1+\frac{\epsilon}{2}}{n}$ and $p_2 = \frac{p-p_1}{1-p_1} \geq \frac{\epsilon}{2n}$. Let us fix $c = \frac{\alpha}{\log \frac{1}{\epsilon}}$for some sufficiently small $\alpha>0$ that we will choose later.

By standard results, see for example \cite[Lemma 5.4]{Janson}, whp the largest component in $G(n,p_1)$ has size at least $\frac{\epsilon}{2} n$ and so whp $G(n,p_1)$ contains a tree $T$ of size $m = \frac{\epsilon}{2} n$. We claim there are at most $\binom{m}{i}2^i$ many partitions $(A,B)$ of $V(T)$ with exactly $i$ \emph{crossing edges}, i.e., edges in $T$ between $A$ and $B$. Indeed, any partition $(A,B)$ of $V(T)$ induces an orientation of the crossing edges in $T$, by orienting them from the vertex in $A$ to the vertex in $B$, and this orientation determines the partition of $V(T)$. Since each of the $\binom{m-1}{i} \leq \binom{m}{i}$ subsets of $E(T)$ of size $i$ admits at most $2^i$ orientations, the bound follows.

We will also use the following bound on the sum of binomial coefficients, which says for all $m$ and $k \leq \frac{m}{2}$,
\begin{equation}\label{e:entropy}
\sum_{i=1}^k \binom{m}{i} \leq 2^{h\left(\frac{k}{m}\right)m}
\end{equation}
where $h(x) = -x \log_2 x - (1-x) \log_2 (1-x)$ is the binary entropy function. We note that 
\begin{equation}\label{e:approx}
h(x) \approx x \log_2 \frac{1}{x}
\end{equation}
for sufficiently small $x$.

It follows that the total number of partitions with at most $c \epsilon^3 n$ crossing edges in $T$ is at most
\begin{align*}
\sum_{i=1}^{c\epsilon^3 n} \binom{m}{i}2^i &\leq 2^{ c\epsilon^3 n}\sum_{i=1}^{c\epsilon^3 n} \binom{m}{i}\\
&\leq 2^{c\epsilon^3 n +h\left(\frac{c\epsilon^3 n}{m}\right) m } \tag{by \eqref{e:entropy}}\\
&\leq 2^{c\epsilon^3 n +h\left(2c\epsilon^2\right) \frac{\epsilon n}{2}} \tag{since $m=\frac{\epsilon n}{2}$}\\
&\leq 2^{c\epsilon^3 n + 3c \epsilon ^3 n \log \left(\frac{1}{\epsilon} \right)} \tag{by \eqref{e:approx}}\\
&\leq 2^{4\alpha \epsilon ^3 n } \tag{since $c = \frac{\alpha}{\log \frac{1}{\epsilon}}$}.
\end{align*}

Now, if we fix a partition $V(T) = V_1 \cup V_2$ with $ |V_1|,|V_2| \geq \frac{m}{3}$, then there are $|V_1|\cdot |V_2| \geq \frac{2m^2}{9} = \frac{\epsilon^2 n^2}{18}$ potential edges between $V_1$ and $V_2$, and we expect at least $\frac{\epsilon^3}{36} n$ of them to appear in $G(n,p_2)$, see Figure \ref{f:crossingedges}. Hence, by Lemma \ref{l:chernoff} the probability that fewer than $c\epsilon^3 n$ of these edges are in $G(n,p_2)$ is at most $ \exp \left( -3 \alpha \epsilon^3n\right)$, for $\alpha$ sufficiently small (independent of $\epsilon$).

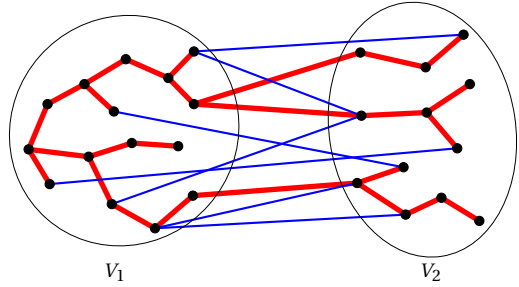
\begin{figure}[!ht]
    \centering
\definecolor{qqqqff}{rgb}{0.,0.,1.}
\definecolor{ffqqqq}{rgb}{1.,0.,0.}
\begin{tikzpicture}[line cap=round,line join=round,>=triangle 45,x=1.0cm,y=1.0cm, scale=0.7]
\clip(1.,-0.5) rectangle (12.,5.);
\draw [line width=2.pt,color=ffqqqq] (2.22,3.)-- (1.86,2.14);
\draw [line width=2.pt,color=ffqqqq] (2.22,3.)-- (2.92,3.38);
\draw [line width=2.pt,color=ffqqqq] (2.92,3.38)-- (3.705533584248081,3.853497516512988);
\draw [line width=2.pt,color=ffqqqq] (2.92,3.38)-- (3.48,2.86);
\draw [line width=2.pt,color=ffqqqq] (3.705533584248081,3.853497516512988)-- (4.512600720534274,3.4968864562935082);
\draw [line width=2.pt,color=ffqqqq] (4.512600720534274,3.4968864562935082)-- (5.,3.);
\draw [line width=2.pt,color=ffqqqq] (4.512600720534274,3.4968864562935082)-- (5.,4.);
\draw [line width=2.pt,color=ffqqqq] (5.,3.)-- (8.16,3.98);
\draw [line width=2.pt,color=ffqqqq] (5.,3.)-- (8.18,2.78);
\draw [line width=2.pt,color=ffqqqq] (8.16,3.98)-- (9.4,3.7);
\draw [line width=2.pt,color=ffqqqq] (9.4,3.7)-- (10.12,4.32);
\draw [line width=2.pt,color=ffqqqq] (8.18,2.78)-- (9.42,2.84);
\draw [line width=2.pt,color=ffqqqq] (9.42,2.84)-- (10.24,3.38);
\draw [line width=2.pt,color=ffqqqq] (9.42,2.84)-- (10.,2.16);
\draw [line width=2.pt,color=ffqqqq] (1.86,2.14)-- (2.26,1.48);
\draw [line width=2.pt,color=ffqqqq] (1.86,2.14)-- (3.,2.);
\draw [line width=2.pt,color=ffqqqq] (3.,2.)-- (3.82,2.26);
\draw [line width=2.pt,color=ffqqqq] (3.,2.)-- (3.44,1.1);
\draw [line width=2.pt,color=ffqqqq] (3.82,2.26)-- (4.7,2.2);
\draw [line width=2.pt,color=ffqqqq] (3.44,1.1)-- (4.26,0.64);
\draw [line width=2.pt,color=ffqqqq] (4.26,0.64)-- (4.98,1.26);
\draw [line width=2.pt,color=ffqqqq] (4.98,1.26)-- (8.1,1.5);
\draw [line width=2.pt,color=ffqqqq] (8.1,1.5)-- (9.02,0.9);
\draw [line width=2.pt,color=ffqqqq] (8.1,1.5)-- (8.98,1.8);
\draw [line width=2.pt,color=ffqqqq] (9.02,0.9)-- (9.7,1.22);
\draw [line width=2.pt,color=ffqqqq] (9.7,1.22)-- (10.42,0.78);
\draw [line width=0.8pt,color=qqqqff] (5.,4.)-- (8.18,2.78);
\draw [line width=0.8pt,color=qqqqff] (3.44,1.1)-- (8.18,2.78);
\draw [line width=0.8pt,color=qqqqff] (3.48,2.86)-- (8.98,1.8);
\draw [line width=0.8pt,color=qqqqff] (4.26,0.64)-- (8.1,1.5);
\draw [line width=0.8pt,color=qqqqff] (4.26,0.64)-- (9.02,0.9);
\draw [line width=0.8pt,color=qqqqff] (5.,4.)-- (10.12,4.32);
\draw [line width=0.8pt,color=qqqqff] (2.26,1.48)-- (10.,2.16);
\draw [rotate around={48.65222278030447:(3.674758663484239,2.4956702534489486)},line width=0.2pt] (3.674758663484239,2.4956702534489486) ellipse (2.228923287521323cm and 2.139491882284884cm);
\draw [rotate around={-81.27830220201824:(9.338551900380512,2.514135205907255)},line width=0.2pt] (9.338551900380512,2.514135205907255) ellipse (2.434464778630063cm and 1.7704223324509234cm);
\begin{scriptsize}
\draw [fill=black] (2.22,3.) circle (2.5pt);
\draw [fill=black] (1.86,2.14) circle (2.5pt);
\draw [fill=black] (2.92,3.38) circle (2.5pt);
\draw [fill=black] (3.705533584248081,3.853497516512988) circle (2.5pt);
\draw [fill=black] (3.48,2.86) circle (2.5pt);
\draw [fill=black] (4.512600720534274,3.4968864562935082) circle (2.5pt);
\draw [fill=black] (5.,3.) circle (2.5pt);
\draw [fill=black] (5.,4.) circle (2.5pt);
\draw [fill=black] (8.16,3.98) circle (2.5pt);
\draw [fill=black] (8.18,2.78) circle (2.5pt);
\draw [fill=black] (9.4,3.7) circle (2.5pt);
\draw [fill=black] (10.12,4.32) circle (2.5pt);
\draw [fill=black] (9.42,2.84) circle (2.5pt);
\draw [fill=black] (10.24,3.38) circle (2.5pt);
\draw [fill=black] (10.,2.16) circle (2.5pt);
\draw [fill=black] (2.26,1.48) circle (2.5pt);
\draw [fill=black] (3.,2.) circle (2.5pt);
\draw [fill=black] (3.82,2.26) circle (2.5pt);
\draw [fill=black] (3.44,1.1) circle (2.5pt);
\draw [fill=black] (4.7,2.2) circle (2.5pt);
\draw [fill=black] (4.26,0.64) circle (2.5pt);
\draw [fill=black] (4.98,1.26) circle (2.5pt);
\draw [fill=black] (8.1,1.5) circle (2.5pt);
\draw [fill=black] (9.02,0.9) circle (2.5pt);
\draw [fill=black] (8.98,1.8) circle (2.5pt);
\draw [fill=black] (9.7,1.22) circle (2.5pt);
\draw [fill=black] (10.42,0.78) circle (2.5pt);
\draw (3.5,-0.2) node {$V_1$};
\draw (9.5,-0.2) node {$V_2$};
\end{scriptsize}
\end{tikzpicture}
    \caption{Whp for every partition $(V_1,V_2)$ of the vertex set of the tree $T \subseteq G(n,p_1)$, which we draw in red, with few crossing edges in $T$, there are many crossing edges in $G(n,p_2)$, drawn in blue.}
    \label{f:crossingedges}
\end{figure}

Hence, by the union bound, the probability that there exists some partition of $V(T)$ with fewer than $c\epsilon^3 n$ crossing edges in $T \cup G(n,p_2) \subseteq G(n,p)$ is at most $\exp \left( - \alpha \epsilon^3n\right) = o(1)$. In particular, it follows that 
\begin{equation}\label{e:bisectionwidth}
\text{whp the bisection width of $H = G(n,p)[V(T)]$ satisfies $b(H) \geq c\epsilon^3 n$.}
\end{equation}

We would like to have a bound on the maximum degree of $G(n,p)$ in order to apply Lemma \ref{l:rank}, but the naive bound of $O(\log n)$ will not be sufficient. However, we note that by Lemma \ref{l:highdegree} there are very few edges in $G(n,p)$ incident with vertices of high degree. Indeed, the number of edges incident with the set $X'$ of the $\frac{\delta}{\log \frac{1}{\delta}}n$ vertices of highest degree in $G(n,p)$, where $\delta$ is sufficiently small, is whp at most $\delta n$ by Lemma \ref{l:highdegree}. In particular, at least one vertex in $X$ must have degree at most $2 \log \frac{1}{\delta}$, since otherwise there would be too many edges incident with $X'$. So, taking $\delta = \frac{c \epsilon^3}{2}$, we see that if we let $X \subseteq X'$ be the set of vertices of degree at least
\[
M = 2 \log \frac{2}{c\epsilon^3} \leq 10 \log \frac{1}{\epsilon},
\]
then whp the number of edges incident with $X$ in $G(n,p)$ is at most $\frac{c\epsilon^3}{2} n$.
Then, given any partition  $(W_1,W_2)$ of $W = V(H)$ such that $|W_1|,|W_2| \geq \frac{m}{3}$ we have 
\[
e(W_1,W_2) \geq b(H) \geq c\epsilon^3 n.
\]
Furthermore, the number of edges incident with $X$ is at most $\frac{c\epsilon^3}{2} n$, and so if we let $W'_i = W_i \setminus X$ for $i=1,2$, then we have $e(W_1',W_2')\geq \frac{c\epsilon^3}{2} n$.

Then, the adjacency matrix $N_{W_1',W_2'}$ between $W_1'$ and $W_2'$ has at least $\frac{c\epsilon^3}{2} n$ non-zero entries and at most $M$ non-zero entries per row, and hence by Lemma \ref{l:rank} it follows that
\[
\rho_H(W_1,W_2) \geq \rho_H(W_1',W_2') \geq \frac{c\epsilon^3}{2M^2} n.
\]

It follows from Lemma \ref{l:smallcut} that $\textrm{rw}(H)\geq \frac{c\epsilon^3}{2M^2} n$ and hence, since $H$ is an induced subgraph of $G(n,p)$, by \eqref{e:vtxminor},
\[
\textrm{rw}(G(n,p))\geq  \textrm{rw}(H)\geq  \frac{c\epsilon^3}{2M^2} n = \Omega\left(\frac{\epsilon^3}{ \left(\log \frac{1}{\epsilon}\right)^3} n \right).
\]
\end{proof}

In order to argue about the tree-width of $G(n,p)$, it will not be sufficient to just find a large tree in $G(n,p)$ in the first sprinkling step. In order to bound the number of balanced separators efficiently, we will need to control the maximum degree of this tree.

For this reason, we will need the following result, asserting the existence of a large bounded degree tree in a supercritical random graph. The proof is relatively standard, and so we relegate it to Appendix \ref{a:largetree}

\begin{theorem}\label{t:largetree}
Let $\delta >0$ be sufficiently small, let $p = \frac{1 + \delta}{n}$. Then whp there exists a tree $T$ in $G(n,p)$ such that $|T| = \Omega(\delta n)$ and $\Delta(T) = O\left(\log \frac{1}{\delta}\right)$.
\end{theorem}

\begin{proof}[Proof of Theorem \ref{t:treewidthdirect}]
We argue via a sprinkling argument with $p_1 = \frac{1+ \frac{\epsilon}{2}}{n}$ and $p_2 = \frac{p-p_1}{1-p_1} \geq \frac{\epsilon}{2n}$. By Theorem \ref{t:largetree} there exist constants $c_1,c_2 >0$ such that whp there exists a tree $T$ in $G(n,p_1)$ with vertex set $V(T)$ such that $|T| = c_1\epsilon n$ and $\Delta(T) \leq c_2 \log \frac{1}{\epsilon}$. Let us also set $c := \frac{c_3}{\log \frac{1}\epsilon}$, where we will choose $c_3>0$ sufficiently small later

We now sprinkle onto the edges of $V(T)$ with probability $p_2$, and claim that whp there are no $\left( c\epsilon^3 n, \frac{1}{2}\right)$-balanced partitions of $G(n,p)[V(T)]$. Given a set $S$ of $c\epsilon^3 n$ vertices in $T$, there are at most $\Delta(T)c\epsilon^3n$ components of $T \setminus S$ and so at most $2^{ \Delta(T)c\epsilon^3 n}$ partitions $(A,S,B)$ of $V(T)$ giving rise to a separation of $T$ with separator $S$, even without considering which are balanced.

For each $\left(c\epsilon^3 n, \frac{1}{2}\right)$-balanced partition $(A,S,B)$ of $T$, since 
\[
\frac{n}{3}\left( c_1 \epsilon- c\epsilon^3 \right) \leq |A|,|B| \leq \frac{2n}{3}\left(c_1 \epsilon - c\epsilon^3 \right) \qquad \text{ and } \qquad |A|+|B| = n(c_1\epsilon - c\epsilon^3),
\]
there are $|A|\cdot |B| \geq\frac{2n^2}{9} \left(c_1 \epsilon - c\epsilon^3\right)^2$ edges between $A$ and $B$, and so the probability that none of these edges are present after sprinkling is at most
\[
(1-p_2)^{\frac{2n^2}{9} \left(c_1 \epsilon - c\epsilon^3\right)^2} \leq \exp \left( -\frac{\epsilon n}{9}\left(c_1 \epsilon - c\epsilon^3\right)^2\right) \leq  \exp \left( -\frac{c_1^2 \epsilon^3 n}{36}\right) 
\]
since $c\epsilon^3 \leq \frac{c_1 \epsilon}{2}$.

In particular, the expected number of $\left(c\epsilon^3 n, \frac{1}{2}\right)$-balanced partitions of $G(n,p)[V(T)]$ will be at most
\begin{align*}
\binom{n}{c\epsilon^3 n} 2^{ \Delta(T)c\epsilon^3n}  \exp \left( -\frac{c_1^2 \epsilon^3 n}{36}\right)  &\leq \left( \frac{e}{c\epsilon^3}\right)^{c\epsilon^3n} \exp \left( \Delta(T) c\epsilon^3n - \frac{c_1^2 \epsilon^3 n}{36}\right)\\
&\leq 
\exp \left( n\left( c\epsilon^3\left(\log \frac{1}{c\epsilon^3} + 1\right) + c_2 \log \frac{1}{\epsilon} c\epsilon^3 - \frac{c_1^2 \epsilon^3}{36}\right) \right)\\
&= \exp \left( - \Omega\left(\epsilon^3 n\right)\right),
\end{align*}
as long as $c_3$ is sufficiently small.

Since whp there are no $\left(c\epsilon^3n, \frac{1}{2}\right)$-balanced partitions of $G(n,p)[V(T)]$, it follows from Lemma \ref{l:twsep} that whp \[
\textrm{tw}(G(n,p)) \geq \textrm{tw}(G(n,p)[V(T)]) \geq c\epsilon^3n = \Omega\left(\frac{ \epsilon^3 n}{\log \frac{1}{\epsilon}} \right).
\]
\end{proof}

\section{The weakly supercritical regime: Proof of Theorem \ref{t:weaklysup}}\label{s:weakly}

In the weakly subcritical regime, we can in fact show a better bound on the rank- and tree-width of $G(n,p)$ by considering more carefully the expansion properties of the giant component $L_1$. We note however, that a naive analogue of Theorem \ref{t:expansion} will not hold in this regime of $p$, since the giant component is likely too sparse to contain a large expanding \emph{subgraph}. However, it will still whp contain a large expanding substructure whose existence will bound the rank-width of $G(n,p)$ from below. We say that a graph $H$ is an \emph{induced topological minor} of $G$ if there is an induced subgraph $K$ of $G$ which is a subdivision of $H$.
\begin{theorem}\label{t:weakexpander}
Let $\epsilon = \epsilon(n)>0$ be such that $\epsilon^3n \rightarrow \infty$ and $\epsilon = o(1)$, and let $p=\frac{1+\epsilon}{n}$. Then there exist constants $\alpha,\delta >0$ such that whp $G(n,p)$ contains some graph $H$ as an induced topological minor such that
\begin{enumerate}
    \item $|V(H)| \geq \delta \epsilon^3 n$;
    \item $H$ is a $3$-regular $\alpha$-expander.
\end{enumerate}
\end{theorem}
We note that it is likely that this could also be derived from work of Ding, Kim, Lubetzky and Peres \cite{Ding}, who give a remarkable description of a simple model contiguous to the giant component in the weakly supercritical regime. However, our exposition is relatively straightforward and short, and so we opt for a direct argument.

To prove Theorem \ref{t:weakexpander} we will use some structural properties of the \emph{$2$-core} of $L_1$.
The $2$-core of a graph is the unique maximal subgraph of minimum degree at least two. The \emph{kernel} of a graph is the graph obtained from the $2$-core by deleting all isolated cycles and contracting all \emph{bare paths}, paths in which all internal vertices have degree two, see Figure \ref{f:corekernel}. 

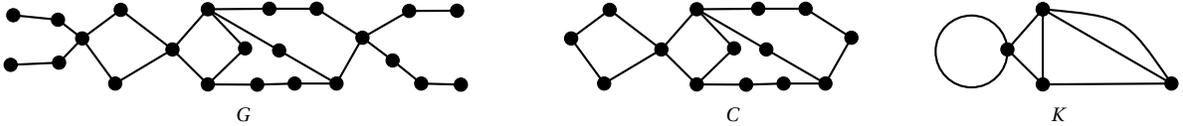
\begin{figure}[!ht]
    \centering
\definecolor{ududff}{rgb}{0.30196078431372547,0.30196078431372547,1.}
\begin{tikzpicture}[line cap=round,line join=round,>=triangle 45,x=1.0cm,y=1.0cm]
\clip(0.,1.3) rectangle (16.,3.2);
\draw [line width=0.8pt] (1.8420245093703773,2.991959873721575)-- (1.3298045767661724,2.611453638072737);
\draw [line width=0.8pt] (1.3298045767661724,2.611453638072737)-- (1.7688502332840623,2.0114245741649532);
\draw [line width=0.8pt] (1.7688502332840623,2.0114245741649532)-- (2.5298627045817383,2.4651050859001065);
\draw [line width=0.8pt] (1.8420245093703773,2.991959873721575)-- (2.5298627045817383,2.4651050859001065);
\draw [line width=0.8pt] (2.5298627045817383,2.4651050859001065)-- (3.,3.);
\draw [line width=0.8pt] (3.,3.)-- (3.495763148921096,2.4797399411173697);
\draw [line width=0.8pt] (3.,2.)-- (3.495763148921096,2.4797399411173697);
\draw [line width=0.8pt] (3.,2.)-- (2.5298627045817383,2.4651050859001065);
\draw [line width=0.8pt] (3.,2.)-- (3.6567465563109893,1.9967897189476902);
\draw [line width=0.8pt] (3.6567465563109893,1.9967897189476902)-- (4.1543316336979315,2.0114245741649532);
\draw [line width=0.8pt] (4.1543316336979315,2.0114245741649532)-- (4.710456131953925,2.0114245741649532);
\draw [line width=0.8pt] (3.,3.)-- (4.710456131953925,2.0114245741649532);
\draw [line width=0.8pt] (3.,3.)-- (3.817729963700882,3.006594728938837);
\draw [line width=0.8pt] (3.817729963700882,3.006594728938837)-- (4.447028738043191,3.006594728938837);
\draw [line width=0.8pt] (4.447028738043191,3.006594728938837)-- (5.055702576893261,2.620915242143432);
\draw [line width=0.8pt] (5.055702576893261,2.620915242143432)-- (4.710456131953925,2.0114245741649532);
\draw [line width=0.8pt] (5.055702576893261,2.620915242143432)-- (5.676356576293283,2.977325018504311);
\draw [line width=0.8pt] (5.676356576293283,2.977325018504311)-- (6.313619434637541,2.9803200586417984);
\draw [line width=0.8pt] (5.055702576893261,2.620915242143432)-- (5.456833748034338,2.3187565337274756);
\draw [line width=0.8pt] (5.456833748034338,2.3187565337274756)-- (5.837339983683176,2.0114245741649524);
\draw [line width=0.8pt] (5.837339983683176,2.0114245741649524)-- (6.364194771504644,1.9967897189476893);
\draw [line width=0.8pt] (1.3298045767661724,2.611453638072737)-- (1.0078377619863865,2.860246176766207);
\draw [line width=0.8pt] (1.3298045767661724,2.611453638072737)-- (1.0224726172036496,2.2894868232929495);
\draw [line width=0.8pt] (1.0078377619863865,2.860246176766207)-- (0.41339036378937305,2.9204192558920705);
\draw [line width=0.8pt] (1.0224726172036496,2.2894868232929495)-- (0.37853898764407756,2.2602171128584234);
\draw [line width=0.8pt] (8.342024509370365,2.991959873721575)-- (7.829804576766161,2.611453638072737);
\draw [line width=0.8pt] (7.829804576766161,2.611453638072737)-- (8.26885023328405,2.0114245741649532);
\draw [line width=0.8pt] (8.26885023328405,2.0114245741649532)-- (9.029862704581722,2.4651050859001065);
\draw [line width=0.8pt] (8.342024509370365,2.991959873721575)-- (9.029862704581722,2.4651050859001065);
\draw [line width=0.8pt] (9.029862704581722,2.4651050859001065)-- (9.5,3.);
\draw [line width=0.8pt] (9.5,3.)-- (9.995763148921077,2.4797399411173697);
\draw [line width=0.8pt] (9.5,2.)-- (9.995763148921077,2.4797399411173697);
\draw [line width=0.8pt] (9.5,2.)-- (9.029862704581722,2.4651050859001065);
\draw [line width=0.8pt] (9.5,2.)-- (10.156746556310969,1.9967897189476902);
\draw [line width=0.8pt] (10.156746556310969,1.9967897189476902)-- (10.65433163369791,2.0114245741649532);
\draw [line width=0.8pt] (10.65433163369791,2.0114245741649532)-- (11.210456131953903,2.0114245741649532);
\draw [line width=0.8pt] (9.5,3.)-- (11.210456131953903,2.0114245741649532);
\draw [line width=0.8pt] (9.5,3.)-- (10.31772996370086,3.006594728938837);
\draw [line width=0.8pt] (10.31772996370086,3.006594728938837)-- (10.94702873804317,3.006594728938837);
\draw [line width=0.8pt] (10.94702873804317,3.006594728938837)-- (11.555702576893239,2.620915242143432);
\draw [line width=0.8pt] (11.555702576893239,2.620915242143432)-- (11.210456131953903,2.0114245741649532);
\draw [line width=0.8pt] (13.629862704581706,2.4651050859001065)-- (14.1,3.);
\draw [line width=0.8pt] (14.1,2.)-- (13.629862704581706,2.4651050859001065);
\draw [line width=0.8pt] (14.1,3.)-- (15.810456131953886,2.0114245741649532);
\draw [line width=0.8pt] (13.152758762539579,2.4384267500454246) circle (0.47784924936229933cm);
\draw [line width=0.8pt] (14.1,2.)-- (15.810456131953886,2.0114245741649532);
\draw [line width=0.8pt] (14.1,3.)-- (14.1,2.);
\draw [line width=0.8pt]  (14.1,3.) .. controls (15.2,2.9) .. (15.810456131953886,2.0114245741649532);
\begin{scriptsize}
\draw [fill=black] (1.8420245093703773,2.991959873721575) circle (2.5pt);
\draw [fill=black] (1.3298045767661724,2.611453638072737) circle (2.5pt);
\draw [fill=black] (1.7688502332840623,2.0114245741649532) circle (2.5pt);
\draw [fill=black] (2.5298627045817383,2.4651050859001065) circle (2.5pt);
\draw [fill=black] (3.,3.) circle (2.5pt);
\draw [fill=black] (3.495763148921096,2.4797399411173697) circle (2.5pt);
\draw [fill=black] (3.,2.) circle (2.5pt);
\draw [fill=black] (3.6567465563109893,1.9967897189476902) circle (2.5pt);
\draw [fill=black] (4.1543316336979315,2.0114245741649532) circle (2.5pt);
\draw [fill=black] (4.710456131953925,2.0114245741649532) circle (2.5pt);
\draw [fill=black] (3.9480750760336605,2.4520504183039655) circle (2.5pt);
\draw [fill=black] (3.817729963700882,3.006594728938837) circle (2.5pt);
\draw [fill=black] (4.447028738043191,3.006594728938837) circle (2.5pt);
\draw [fill=black] (5.055702576893261,2.620915242143432) circle (2.5pt);
\draw [fill=black] (5.676356576293283,2.977325018504311) circle (2.5pt);
\draw [fill=black] (6.313619434637541,2.9803200586417984) circle (2.5pt);
\draw [fill=black] (5.456833748034338,2.3187565337274756) circle (2.5pt);
\draw [fill=black] (5.837339983683176,2.0114245741649524) circle (2.5pt);
\draw [fill=black] (6.364194771504644,1.9967897189476893) circle (2.5pt);
\draw [fill=black] (1.0078377619863865,2.860246176766207) circle (2.5pt);
\draw [fill=black] (1.0224726172036496,2.2894868232929495) circle (2.5pt);
\draw [fill=black] (0.41339036378937305,2.9204192558920705) circle (2.5pt);
\draw [fill=black] (0.37853898764407756,2.2602171128584234) circle (2.5pt);
\draw [fill=black] (8.342024509370365,2.991959873721575) circle (2.5pt);
\draw [fill=black] (7.829804576766161,2.611453638072737) circle (2.5pt);
\draw [fill=black] (8.26885023328405,2.0114245741649532) circle (2.5pt);
\draw [fill=black] (9.029862704581722,2.4651050859001065) circle (2.5pt);
\draw [fill=black] (9.5,3.) circle (2.5pt);
\draw [fill=black] (9.995763148921077,2.4797399411173697) circle (2.5pt);
\draw [fill=black] (9.5,2.) circle (2.5pt);
\draw [fill=black] (10.156746556310969,1.9967897189476902) circle (2.5pt);
\draw [fill=black] (10.65433163369791,2.0114245741649532) circle (2.5pt);
\draw [fill=black] (11.210456131953903,2.0114245741649532) circle (2.5pt);
\draw [fill=black] (10.424564281702578,2.46563871974202) circle (2.5pt);
\draw [fill=black] (10.31772996370086,3.006594728938837) circle (2.5pt);
\draw [fill=black] (10.94702873804317,3.006594728938837) circle (2.5pt);
\draw [fill=black] (11.555702576893239,2.620915242143432) circle (2.5pt);
\draw [fill=black] (13.629862704581706,2.4651050859001065) circle (2.5pt);
\draw [fill=ududff] (14.1,3.) circle (2.5pt);
\draw [fill=black] (14.1,2.) circle (2.5pt);
\draw [fill=black] (15.810456131953886,2.0114245741649532) circle (2.5pt);
\draw [fill=black] (14.1,3.) circle (2.5pt);
\draw (3.47,1.6) node {$G$};
\draw (9.98,1.6) node {$C$};
\draw (14.31,1.6) node {$K$};
\end{scriptsize}
\end{tikzpicture}
    \caption{A graph $G$, its 2-core $C$, and its kernel $K$.}
    \label{f:corekernel}
\end{figure}

It is easy to see that if we condition on the degree sequence of the $2$-core $C$ of $L_1$, then $C$ is uniformly distributed over all simple graphs with this degree sequence. However, the degree distribution of the $2$-core is well understood. The following bounds follow from work of {\L}uczak \cite{Luczakcycle}.

\begin{lemma}\label{l:luczak}
Let $p=\frac{1+\epsilon}{n}$ be such that $\epsilon=o(1)$ and $\epsilon^3 n \to \infty$, and for $i\geq 2$, let $D_i$ be the random variable which counts the number of vertices of degree $i$ in the $2$-core of the largest component $L_1$ of $G(n,p)$. Then whp
\begin{enumerate}[(i)]
\item $D_2= (1+o(1))2\epsilon^2n$;
    \item\label{i:degthree} $D_3= (1+o(1))\frac{4}{3}\epsilon^3n$;
    \item\label{i:degfour} $\sum_{i\geq 4} i D_i = o(\epsilon^3 n)$;
    \item\label{i:lambda} $\lambda = \frac{\sum_i D_i i(i-1)}{2 \sum_i D_i i} = O(1)$.
\end{enumerate}
\end{lemma}


Hence, it follows from Lemma \ref{l:luczak} \ref{i:lambda} and \cite[Theorem 10.3]{Frieze} that any property which holds whp in the configuration model for any degree sequence ${\bf d}$ satisfying the conclusions of Lemma \ref{l:luczak}, also holds whp in the $2$-core. 

The genesis of the following argument already appears in work of {\L}uczak \cite{Luczakcycle}, although he does not make the following claim explicit.

\begin{lemma}\label{l:indtopmin}
Let ${\bf d} \in \mathbb{N}^s$ be a degree sequence such that $\min_i d_i \geq 2$. Then $G^*({\bf d})$ can be coupled with some pair $(m,G^*(m,3))$ such that with probability one
\begin{enumerate}[(i)]
\item $m \geq D_3 - \sum_{i \geq 4} i D_i$;
\item $G^*(m,3)$ is an induced topological minor of $G^*({\bf d})$. 
\end{enumerate}
\end{lemma}
\begin{proof}
The idea behind the construction is simple. We expose our multigraph $G^*({\bf d})$ and first delete all the vertices of degree at least four. We then take the $2$-core of this graph, which we can form by recursively deleting leaves, leaving us with an induced subgraph in which all degrees are two or three, which will contain the desired cubic graph as an induced topological minor after contracting all maximal bare paths. It remains to show that the subgraph we obtain in this way is distributed as $G^*(m,3)$ for some appropriate $m$.

In order to show this, we will utilise the \emph{principle of deferred decisions} to partially expose the configuration giving rise to $G^*({\bf d})$, allowing us to assume that the remaining part of the configuration is still uniformly distributed on the unmatched half-edges.

Let $\mathcal{W}({\bf d}) = \{W_i \colon i \in [s]\}$ be the set of cells used to construct $G^*({\bf d})$. Given a partial matching $E$ of the half-edges in $W := \bigcup_{i \in [s]} W_i$ we say that a half-edge $e \in W_i$ is a \emph{leaf with respect to $E$} if it is the only half-edge in $W_i$ which is unmatched in $E$.

We perform the following steps:
\begin{enumerate}[S1)]
    \item\label{i:one} We expose the matching edges containing each half-edge which is contained in a cell of degree at least $4$. Let $M_1$ be the set of edges contained in this partial matching.
    \item\label{i:two} We initially set $M_2 = M_1$ and then recursively do the following, as long as there is some half-edge $e$ which is a leaf with respect to $M_2$: Let $e$ be a leaf with respect to $M_2$. We expose the matching edge containing $e$ and add this edge to $M_2$.
\end{enumerate}

Let $\mathcal{W}_1 =\{ W_i \setminus \bigcup M_2 \colon i \in [s] \}$ and let $\mathcal{W}_2 = \{W'_1,\ldots, W'_r\}$ be the set of non-empty cells in $\mathcal{W}_1$. We note the following properties of $\mathcal{W}_2$:
\begin{enumerate}[a)]
    \item\label{i:degree} Each cell in $\mathcal{W}_2$ has size two or three;
    \item\label{i:cubic} The number of cells of degree three in $\mathcal{W}_2$ is at least $D_3 - \sum_{i \geq 4} i D_i$;
    \item\label{i:uniform} Conditioned on the outcome $M_2$ of Steps \ref{i:one} and \ref{i:two} the remaining configuration $M' = M \setminus M_2$ is uniformly distributed over all configurations on $\mathcal{W}_2$;
    \item\label{i:subgraph} The graph $G(\mathcal{W}_2,M')$ is an induced subgraph of $G(\mathcal{W},M)$.
\end{enumerate}

Indeed, the first is clear by construction. To see that the second holds, we note that, since $M_1$ contains at most $\sum_{i \geq 4} i D_i$ edges and all cells initially have size at least two, the total number of cells of size three which contain a half-edge in $M_1$ plus the total number of leaves with respect to $M_1$ is at most $\sum_{i \geq 4} i D_i$. However in each recursion step in Step \ref{i:two} this property is maintained for $M_2$, and hence the total number of cells of degree three in $\mathcal{W}_2$ is at least $D_3 - \sum_{i \geq 4} i D_i$. The third holds as mentioned by the principle of deferred decisions, and the fourth holds by the construction of the graphs $G(\mathcal{W}_2,M')$ and $G(\mathcal{W},M)$.

\begin{figure}[!ht]
    \centering
\definecolor{ffqqqq}{rgb}{1.,0.,0.}
\definecolor{ududff}{rgb}{0.30196078431372547,0.30196078431372547,1.}
\definecolor{qqqqff}{rgb}{0.,0.,1.}
\begin{tikzpicture}[line cap=round,line join=round,>=triangle 45,x=1.0cm,y=1.0cm,scale=0.95]
\clip(1.,1.) rectangle (18.,6.8);
\draw [line width=0.8pt,color=qqqqff] (2.0143416878467284,3.7039139741947706) circle (0.2856590801756741cm);
\draw [line width=0.8pt,color=qqqqff] (2.,2.5) circle (0.2856590801756742cm);
\draw [line width=0.8pt,color=qqqqff] (4.,4.9) circle (0.2856590801756711cm);
\draw [line width=0.8pt,color=qqqqff] (4.,3.713399960778924) circle (0.2856590801756742cm);
\draw [line width=0.8pt,color=qqqqff] (4.018999308401302,2.5049612537457993) circle (0.2856590801756735cm);
\draw [line width=0.8pt,color=qqqqff] (6.0071302478802595,3.6719478074947682) circle (0.28565908017567565cm);
\draw [line width=0.8pt,color=qqqqff] (6.,2.5) circle (0.28565908017567554cm);
\draw [line width=0.8pt,color=qqqqff] (8.022219371940848,2.512842452406194) circle (0.28565908017567715cm);
\draw [line width=0.8pt,color=qqqqff] (12.012643789666217,3.650460788331213) circle (0.28565908017567615cm);
\draw [line width=0.8pt,color=qqqqff] (12.,2.5) circle (0.28565908017567654cm);
\draw [line width=0.8pt,color=qqqqff] (14.872306598199117,3.227046747301644) circle (0.28565908017567654cm);
\draw [line width=0.8pt,color=qqqqff] (14.018660003922108,3.7123773129820323) circle (0.28565908017567543cm);
\draw [line width=0.8pt,color=qqqqff] (14.03825994394607,2.5204424965005203) circle (0.28565908017567543cm);
\draw [line width=0.8pt,color=qqqqff] (17.031077998434345,2.488414103880816) circle (0.28565908017567476cm);
\draw [->,line width=0.8pt] (9.50296741660591,3.530694150898469) -- (10.746242181215761,3.547053029380177);
\draw [line width=0.8pt, color=black] (1.46,4.16)-- (1.46,2)-- (2.48,2)-- (2.48,4.16);
\draw [line width=0.8pt, color=black] (2.48,4.16)-- (1.46,4.16);
\draw [line width=0.8pt, color=black] (3.52,5.4)-- (3.52,2)-- (4.5,2)-- (4.5,5.4);
\draw [line width=0.8pt, color=black] (4.5,5.4)-- (3.52,5.4);
\draw [line width=0.8pt, color=black] (5.5,4.25)-- (5.5,2)-- (6.47,2)-- (6.47,4.25);
\draw [line width=0.8pt, color=black] (11.53,4.14)-- (11.53,2)-- (12.46,2)-- (12.46,4.14);
\draw [line width=0.8pt, color=black] (13.45,4.26)-- (13.45,2)-- (15.4,2)-- (15.4,4.26);
\draw [line width=0.8pt, color=black] (16.5,3)-- (16.5,2)-- (17.5,2)-- (17.5,3);
\draw [line width=0.8pt, color=black] (5.5,4.25)-- (6.47,4.25);
\draw [line width=0.8pt, color=black] (11.53,4.14)-- (12.46,4.14);
\draw [line width=0.8pt, color=black] (13.45,4.26)-- (15.4,4.26);
\draw [line width=0.8pt, color=black] (7.53,3)-- (7.53,2.)-- (8.52,2)-- (8.52,3);
\draw [line width=0.8pt, color=black] (7.53,3)-- (8.52,3);
\draw [line width=0.8pt, color=black] (16.5,3)-- (17.5,3);
\draw [line width=0.8pt] (2.299707270515671,3.7168597964871934)-- (3.714687311420669,3.699336539344371);
\draw [line width=0.8pt] (2.2856588772720468,2.499659526086026)-- (3.7343969509671244,2.4804131737858373);
\draw [line width=0.8pt,color=ffqqqq] (4.253353915366476,4.768042038302773)-- (5.726434307855465,3.724965722772624);
\draw [line width=0.8pt] (6.285073162913304,2.481713341655241)-- (7.73661517966188,2.518442045203212);
\draw [line width=0.8pt] (12.298133231703288,3.640617589045638)-- (13.744421431065637,3.6324128505974227);
\draw [line width=0.8pt] (12.28513213604817,2.5173428682582175)-- (13.753119579341123,2.5032354467891285);
\draw [line width=0.8pt] (15.152100552437101,3.169458091659352)-- (16.74671107448447,2.515553791611107);
\begin{scriptsize}
\draw[color=qqqqff] (2.0143416878467284,3.7039139741947706) node {$u$};
\draw[color=qqqqff] (2.,2.5) node {$v$};
\draw[color=qqqqff] (4.,4.9) node {$e$};
\draw[color=qqqqff] (4.,3.713399960778924) node {$w$};
\draw[color=qqqqff] (4.018999308401302,2.5049612537457993) node {$x$};
\draw[color=qqqqff] (6.0071302478802595,3.6719478074947682) node {$f$};
\draw[color=qqqqff] (6.,2.5) node {$y$};
\draw[color=qqqqff] (8.022219371940848,2.512842452406194) node {$z$};
\draw[color=qqqqff] (12.012643789666217,3.650460788331213) node {$u$};
\draw[color=qqqqff] (12.,2.5) node {$v$};
\draw[color=qqqqff] (14.872306598199117,3.227046747301644) node {$y$};
\draw[color=qqqqff] (14.018660003922108,3.7123773129820323) node {$w$};
\draw[color=qqqqff] (14.03825994394607,2.5204424965005203) node {$x$};
\draw[color=qqqqff] (17.031077998434345,2.488414103880816) node {$z$};
\draw[color=black] (2,1.6) node {$V_1$};
\draw[color=black] (4,1.6) node {$V_2$};
\draw[color=black] (6,1.6) node {$V_3$};
\draw[color=black] (8,1.6) node {$V_4$};
\draw[color=black] (12,1.6) node {$V_1$};
\draw[color=black] (14.5,1.6) node {$V_{2,3}$};
\draw[color=black] (17,1.6) node {$V_4$};
\draw[color=black] (3.9991271667776846,6.440791830653201) node {\large $\mathcal{V}$};
\draw[color=black] (15.049885732511992,6.3696114212120785) node {\large $\mathcal{V'}$};
\draw[color=ffqqqq] (5,4.5) node {$m$};
\end{scriptsize}
\end{tikzpicture}
    \caption{The effect of contracting an edge $m$ in a configuration on a set of cells $\mathcal{V}$. The cells $V_2$ and $V_3$ containing the two half-edges $e$ and $f$ matched by $m$ are replaced by a single cell $V_{2,3}$.}
    \label{f:contraction}
\end{figure}
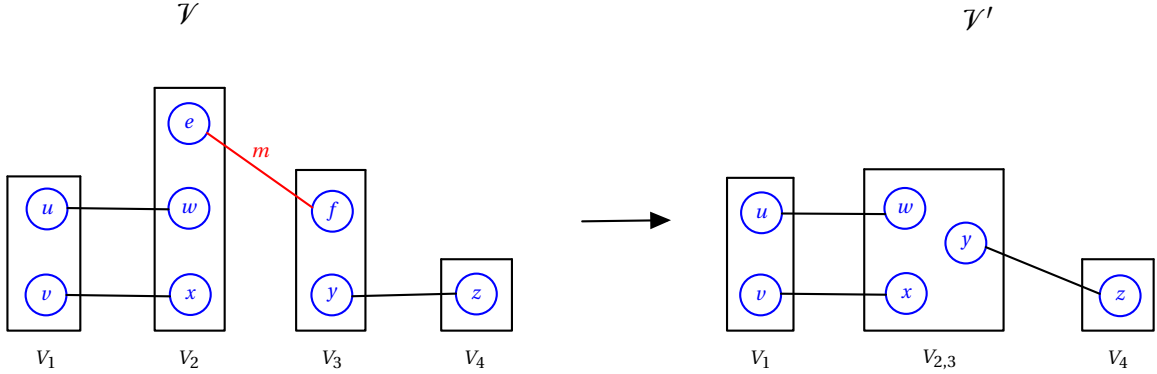

Modelling the contraction of bare paths in the configuration model is a little fiddly. Let $\mathcal{V}$ be an arbitrary collection of cells and let $V_i$ be a cell in $\mathcal{V}$. Let $e$ be an arbitrary half-edge in $V_i$ and let $m=\{e,f\}$ be a matching edge containing $e$, where $f \in V_j$. Let $\mathcal{V}'$ be the set of cells obtained from $\mathcal{V}$ by replacing $V_i$ and $V_j$ with a single cell $V_{ij} = (V_i \cup V_j) \setminus \{e,f\}$, see Figure \ref{f:contraction}. We say $\mathcal{V}'$ is obtained from $\mathcal{V}$ by \emph{contracting} $m$. Then, it is clear that 
\begin{align}
    &\text{for any configuration $M$ on $\mathcal{V}$ containing $m$, the graph $G(\mathcal{V}', M \setminus \{m\})$ is the minor of $G(\mathcal{V},M)$}\nonumber \\
    &\text{obtained by contracting an edge between $i$ and $j$. Furthermore, if the degree of $V_i$ is two,} \nonumber \\
    &\text{then $G(\mathcal{V},M)$ is obtained from $G(\mathcal{V}', M \setminus \{m\})$ by subdividing an edge (Or adding an isolated} \nonumber\\
   &\text{loop, in the case where $m$ is fully contained in $V_i$).} \tag{$\dagger$}\label{e:contract}
\end{align}
We then perform the following final step on $\mathcal{W}_2$:
\begin{enumerate}[S3)]
    \item\label{i:three} Recursively, as long as there is some cell of size two, let $W$ be a cell of size two and let $e$ be an arbitrary half-edge in $W$. We expose the matching edge $m$ containing $e$ and we contract $m$.  
\end{enumerate}

Let the set of contracted edges be $M_3$ and the resulting set of (non-empty) cells be $\mathcal{W}_3$. Since each cell in $\mathcal{W}_2$ has size two or three, each contraction decreases the number of cells of size two by one, but does not change the number of cells of size three. Hence, every cell in $\mathcal{W}_3$ has size three, and  $\mathcal{W}_3$ contains at least $D_3 - \sum_{i \geq 4} i D_i$ many cells. Since $\bigcup \mathcal{W}_3 \cup \bigcup M_3 = \bigcup \mathcal{W}_2$, conditioned on the outcome $M_3$ of Step \ref{i:three} the remaining configuration $M'' = M' \setminus M_3$ is uniformly distributed over all configurations on $\mathcal{W}_3$.

Furthermore, by \eqref{e:contract} it follows that $G(\mathcal{W}_2,M')$ is a subdivision of $G(\mathcal{W}_3,M'')$, and hence $G(\mathcal{W}_3,M'')$ is an induced topological minor of $G(\mathcal{W},M)$. However, since $M''$ is uniformly distributed over all configurations on $\mathcal{W}_3$, each cell of $\mathcal{W}_3$ has size three and there are at least $D_3 - \sum_{i \geq 4} i D_i$ many cells in $\mathcal{W}_3$, the theorem follows.
\end{proof}

It is known that $G^*(m,3)$ has good expansion properties. For example, the following theorem is a consequence of work of Bollob\'{a}s \cite{Bollobasisometric}, although its proof is just a, by now relatively standard, first moment calculation.

\begin{theorem}\label{t:regularexpander}
For any $d \geq 3$ there exists an $\alpha:=\alpha(d) >0$ such that $G^*(m,d)$ is an $\alpha$-expander with probability tending to $1$ as $m \to \infty$.
\end{theorem}

We now have the necessary ingredients to complete the proof of Theorem \ref{t:weakexpander}.
\begin{proof}[Proof of Theorem \ref{t:weakexpander}]
Let $C$ be the $2$-core of the largest component of $G(n,p)$, which we note is an induced subgraph of $G(n,p)$. Let us condition on the degree sequence ${\bf d}$ of $C$, which we may assume satisfies the conclusions of Lemma \ref{l:luczak}. In particular, by the comment after Lemma \ref{l:luczak}, any property which holds in $G^*({\bf d})$ whp also holds in $C$ whp. Let us write $D_i$ for the number of entries in ${\bf d}$ which are equal to $i$.

By Lemma \ref{l:indtopmin} there exists some random variable $m$ such that we can couple $G^*(\bf{d})$ with the pair ${(m, G^*(m,3))}$ such that with probability one,
\[
m \geq D_3 - \sum_{i \geq 4}i D_i =\left(\frac{4}{3} +o(1)\right)\epsilon^3 n := \delta \epsilon^3 n
\]
and $G^*(m,3)$ is an induced topological minor of $G^*(d)$. Since, by Theorem \ref{t:regularexpander}, whp $G^*(m,3)$ is an $\alpha$-expander, where $\alpha:=\alpha(3)$ is as in Theorem \ref{t:regularexpander}, it follows that whp $G^*(\bf{d})$ contains an induced topological minor of a $3$-regular graph on at least $\delta\epsilon^3 n$ vertices which is an $\alpha$-expander. It follows that whp the same is true for $C$ and so, since $C$ is an induced subgraph of $G(n,p)$, the conclusion holds.
\end{proof}

Note that our bounds on $\alpha$ and $\delta$ can be made relatively explicit: Lemma \ref{l:luczak} implies that we can take $\delta \approx \frac{4}{3}$, and $\alpha$ can be chosen to be the appropriate $\alpha(3)$ from Theorem \ref{t:regularexpander}, which is shown in \cite{Bollobasisometric} to be at least $\frac{2}{11}$. The current best known bound is $\alpha(3) \geq \frac{1}{4.95}$, which was shown by Kostochka and Melnikov \cite{Kostochka}.

\begin{theorem}\label{t:rw}
Let $\epsilon = \epsilon(n)>0$ be such that $\epsilon^3n \rightarrow \infty$ and $\epsilon = o(1)$, and let $p=\frac{1+\epsilon}{n}$. Then whp \[
\textrm{rw}(G(n,p)) = \Omega(\epsilon^3 n).
\]
\end{theorem}
\begin{proof}
By Theorem \ref{t:weakexpander} there exist constants $\alpha,\delta >0$ such that whp $G(n,p)$ contains an induced topological minor of a graph $H$ with $|V(H)| \geq \delta \epsilon^3 n:=m$ which is a $3$-regular $\alpha$-expander. By Lemma \ref{l:indtopminorvtxminor}, $H$ is a vertex-minor of $G$ and hence $\textrm{rw}(G) \geq \textrm{rw}(H)$.

Let $S$ be an arbitrary subset of $V(H)$ such that $ \frac{m}{3} \leq |S|\leq \frac{m}{2}$. Since $H$ is a 3-regular $\alpha$-expander it follows that
\[
e_H(S,S^c) \geq \alpha d_H(S) \geq \alpha m.
\]

Hence, by Lemma \ref{l:rank}, we can conclude that $\rho_H(S,S^c) \geq \frac{\alpha m }{9}$. Since $S$ was arbitrary, it is easy to conclude from Lemma \ref{l:smallcut} that 
\[
\textrm{rw}(G) \geq \textrm{rw}(H) \geq \frac{\alpha m }{9} = \frac{\alpha \delta}{9} \epsilon^3 n.
\]
\end{proof}

\begin{corollary}\label{c:tw}
Let $\epsilon = \epsilon(n)>0$ be such that $\epsilon^3n \rightarrow \infty$ and $\epsilon = o(1)$ and let $p=\frac{1+\epsilon}{n}$. Then whp \[
\textrm{tw}(G(n,p)) = \Omega(\epsilon^3 n).
\]
\end{corollary}

\begin{proof}[Proof of Theorem \ref{t:weaklysup}]
The lower bounds follow immediately from Theorem \ref{t:rw} and Corollary \ref{c:tw}.

For the upper bounds, we note that it is relatively easy to see that the tree-width of a connected graph can be bounded from above by one plus its \emph{excess}, the number of edges minus the number of vertices, see \cite[Proposition 4.3]{Serra}. Indeed, every graph with excess $-1$ is a tree, which has tree-width one, and adding a single edge to a graph can increase the tree-width by at most one.

However, it is known, see for example \cite[Theorem 5.13]{Janson}, that if $\epsilon = \epsilon(n)>0$ is such that $\epsilon^3n \rightarrow \infty$ and $\epsilon = o(1)$, and $p=\frac{1+\epsilon}{n}$, then whp the excess of the largest component $L_1$ of $G(n,p)$ is $\Theta(\epsilon^3 n)$, and every other component has excess at most $0$. In particular, it follows that whp
\begin{align*}
\textrm{tw}(G(n,p)) &= \max\{ \textrm{tw}(K) \colon K \text{ a component of } G(n,p) \}\\ 
&\leq \max\{ \textrm{excess}(K) \colon K \text{ a component of } G(n,p) \} \\
&\leq \textrm{excess}(L_1) = O(\epsilon^3 n).
\end{align*}

In particular, it follows that whp $\textrm{tw}(G(n,p)) = O\left(\epsilon^3 n\right)$, and so, by Theorem \ref{t:rwtw}, the corresponding upper bound also holds for the rank-width. 
\end{proof}

\section{Discussion}\label{s:discussion}
As mentioned in the introduction, it is relatively easy to see that the proof of Theorem \ref{t:rankwidthdirect} also works without changes in the weakly supercritical regime. However, for the proof of Theorem \ref{t:treewidthdirect} we used Theorem \ref{t:largetree}, which was only stated for $\epsilon >0$ constant. Nevertheless, we note that a similar result can be seen to hold in the weakly supercritical regime, although since the bounds in Theorem \ref{t:treewidthdirect} are superseded by the bounds in Theorem \ref{t:weaklysup} in this regime, and, whilst the idea is simple, the technical overhead of implementing it is high, we will just give a short sketch of the idea below. Using such a result it can again be seen that the proof of Theorem \ref{t:treewidthdirect} also works in the weakly supercritical regime.

Indeed, standard results on the distribution of the kernel and $2$-core of the largest component of $G(n,p)$ when $p = \frac{1+\delta}{n}$ with $\delta^3 n \to \infty$ and $\delta=o(1)$ imply that whp there is an almost spanning connected subgraph $\tilde{H}$ of the $2$-core, which has order $\Omega(\delta^2 n)$ and maximum degree three. 

Furthermore, if we let $L_1$ be the giant component of $G(n,p)$ and $C$ be the $2$-core of $L_1$, then it follows from work of Ding, Kim, Lubetzky and Peres \cite{Ding} that the graph $R:= L_1 - E(C)$ is contiguous to a graph built by choosing, for each vertex $v \in C$ independently, a tree $T_v$ which is distributed as a Poisson$(1-\delta')$ Galton-Watson tree, where $\delta'$ is the smallest solution to $(1-\delta')e^{-(1-\delta')} = (1+\delta)e^{1+\delta}$, and satisfies $\delta' = 1 - \delta + O(\delta^2)$. Hence, standard bounds on the distribution of Galton-Watson trees imply that with a positive probability the order of $T_v$ is $\Omega\left( \frac{1}{\delta}\right)$, and whp the maximum degree in $T_v$ is $O\left( \frac{1}{\log \frac{1}{\delta}} \right)$, and so whp a positive proportion of the trees $T_v$ with roots in $\tilde{H}$ will satisfy these two properties. Taking the union of these trees together with $\tilde{H}$ itself will lead to the claimed subgraph.

Whilst Theorem \ref{t:weaklysup} is optimal in terms of its dependence on $\epsilon$ and $n$, it seems unlikely that Theorems \ref{t:treewidthdirect} and \ref{t:rankwidthdirect} are. In fact, it seems reasonable to expect that the methods we use to prove Theorem \ref{t:weaklysup} should extend to prove a similar bound in the barely supercritical regime. Indeed, it is likely that the methods of {\L}uczak \cite{Luczak90} can give bounds of a similar nature to Lemma \ref{l:luczak} on the likely degree sequence of the $2$-core of the largest component of $G(n,p)$, in particular that the number of edges incident with vertices of degree four or larger is much smaller than the number of vertices of degree three, which is of order $\Theta(\epsilon^3n)$. It would then follow by the same arguments as in Section \ref{s:weakly} that whp $\textrm{rw}(G(n,p)) = \Omega(\epsilon^3 n)$. However, since the technical details in \cite{Luczak90} are quite involved, we leave this as an open question, and note that perhaps the main interest in Theorems \ref{t:treewidthdirect} and \ref{t:rankwidthdirect} lies in the simplicity and directness of the proofs, rather than their asymptotic dependence on $\epsilon$.

\begin{question}
Let $\epsilon >0$ be a sufficiently small constant and let $p=\frac{1+\epsilon}{n}$. Is there a constant $c>0$ (independent of $n$ and $\epsilon$) such that whp
\[
\textrm{tw}(G(n,p)) \geq c \epsilon^3 n?
\]
\end{question}
We note that, in the barely supercritical regime there is also a natural upper bound on tw$(G(n,p))$ of $O\left(\epsilon^3 n\right)$ given by the likely excess of the giant component, together with the fact that whp all other components are trees or unicyclic.

For large enough $\epsilon$, the fact that rank- and tree-width are vertex-Lipschitz functions imply that these parameters are tightly concentrated about their mean, for example using the Azuma-Hoeffding inequality, see e.g., \cite{Frieze}.
However, for smaller $\epsilon$ is it not clear if these parameters are concentrated in a range of values of size $o\left(\epsilon^3 n\right)$. It would also be interesting to know, if the parameters are tightly concentrated, what the correct leading constant should be.

We note that the upper and lower bounds we get for the rank- and tree-width of $G(n,p)$ in the weakly supercritical regime are not \emph{so} far apart, differing by less than a factor of $100$. Explicitly, using the bounds on $\alpha$ and $\delta$ from Theorem \ref{t:weakexpander}, and the fact that whp the excess of $G(n,p)$ is $(1+o(1)) \frac{2}{3} \epsilon^3 n$ we can deduce that whp
\[
0.02\leq \frac{\textrm{tw}(G)}{\epsilon^3 n} , \frac{\textrm{rw}(G)}{\epsilon^3 n} \leq 1.34.
\]

\begin{question}
Let $\epsilon=\epsilon(n)$ be such that $\epsilon=o(1)$ and $\epsilon^3 n \to \infty$. Is it the case that whp
\[
\textrm{tw}(G(n,p)) = (1+o(1)) c_t \epsilon^3 n \qquad \text{   and   }\qquad  \textrm{rw}(G(n,p)) = (1+o(1)) c_r \epsilon^3 n 
\]
for some constants $c_t, c_r >0$ and can we determine the value of the constants $c_t$ and $c_r$?
\end{question}

Spencer and T\'{o}th \cite{SpencerToth} used the lower bound on the bisection width of a subgraph of $G(n,p)$ from \eqref{e:bisectionwidth} to bound the \emph{crossing number} of $G(n,p)$ in the barely supercritical regime. We say a \emph{drawing} of a graph $G$ is a mapping which assigns to each vertex of $G$ a point in the plane and to each edge a continuous arc, such that no arc passes through a vertex and no three arcs meet at a point. A \emph{crossing point} in a drawing is a point where two arcs meet. Then, the crossing number $\CR (G)$ of a graph $G$ is the minimum number of crossing points in a drawing of $G$.

Leighton \cite{L84} used the well-known Lipton-Tarjan planar separator theorem \cite{LT79} to give a connection between the crossing number and the bisection width $b(G)$ of a graph. Explicitly, it was shown by Pach, Shahrokhi and Szegedy \cite[Theorem 2.1]{PSS96} that for every graph $G$
\begin{equation}\label{e:CRB}
b(G) \leq 2 \sqrt{16 \CR (G) + \sum_{v\in V(G)} d(v)^2}.
\end{equation}

Using this and \eqref{e:bisectionwidth}, Spencer and T\'{o}th \cite{SpencerToth} showed the following:
\begin{theorem}[\cite{SpencerToth}]
Let $d > 1$ and let $p=\frac{d}{n}$. Then
\[
\mathbb{E}\big(\CR (G(n,p))\big) = \Omega\left(n^2\right).
\]
\end{theorem}

We note that, using some of our intermediary results we can extend this into the weakly supercritical regime.

\begin{theorem}
Let $\epsilon:=\epsilon(n) > 0$ be such that $\epsilon^3 n \to \infty$ and $\epsilon = o(1)$, and let $p=\frac{1+\epsilon}{n}$. Then whp
\[
\CR (G(n,p)) = \Omega\left(\epsilon^6 n^2\right).
\]
\end{theorem}
\begin{proof}
We first note that the crossing number is an increasing graph property, and furthermore it is clear that if $H$ is a subdivision of $G$, then $\CR (H) = \CR (G)$. 

Hence, by Theorem \ref{t:weakexpander} there exist $\alpha,\delta >0$ such that whp $\CR (G(n,p)) \geq \CR (H)$ where $H$ is a $3$-regular $\alpha$-expander on $m \geq \delta \epsilon^3 n$ vertices. In particular, it follows from the definition of an $\alpha$-expander that $b(H) \geq \alpha m$.

Therefore, since $H$ is $3$-regular, it follows from \eqref{e:CRB} that
\[
\CR (G) \geq \CR(H) \geq \frac{ b(H)^2 - 4  \sum_{v\in V(H)} d(v)^2}{16} =  \Omega\left(m^2\right) = \Omega\left(\epsilon^6 n^2\right).
\]
\end{proof}

%
%
%
\section*{Acknowledgements}
The first author was supported in part by FWF I3747 and W1230, the second author by FWF P36131 and the third author by FWF W1230. For the purpose of open access, the authors have applied a CC BY public copyright licence to any Author Accepted Manuscript version arising from this submission.

 \bibliographystyle{plain}
\bibliography{References}

\begin{appendix} 
\section{Proof of Theorem \ref{t:largetree}}\label{a:largetree}
Let $G:= G(n,p)$ and let $K:= 4 \log \frac{1}{\delta}$. The basic idea is to analyse a restricted breadth first search process on the graph, where we limit each vertex to have at most $K$ neighbours. In this case, Lemma \ref{l:restricted} will imply that we expect each vertex to have more than one neighbour, and so we should expect this process to grow to a large size. In practice, we have to first build a large tree `by hand', so as to guarantee that the correct growth happens with high probability. This strategy to build a large bounded degree tree already appears in the work of Erde, Kang and Krivelevich \cite{Erde}, where they use similar arguments to build such a tree with large excess in an arbitrary random subgraph.
	
	Throughout the process we will maintain a set of vertices $X$ which will include the vertices of the tree we are building, together with a small set of vertices which we have discarded during an initial step. We will work under the assumption that $|X| \leq \frac{\delta n}{4}$, and later verify that this holds throughout our process. Let us write $U := [n] \setminus X$.
	
 In the initial step, we will build a \emph{partial binary tree} of order $N_1 = 4  \log \log \log n + 1$ via a greedy process. By a partial binary tree we mean a rooted tree, rooted at a vertex $v$ of degree two, in which all other vertices have degree three or one, such that there is some integer $\ell$ such that every leaf is at distance $\ell$ or $\ell-1$ from $v$. Note that, since there are at least $2^{\ell-1}$ leaves in such a tree, and at most $2^{\ell+1}- 1$ vertices, there will be at least $N_2 :=  \log \log \log n $ leaves in such a partial binary tree of order $N_1$.
 
 To build our partial binary tree of order $N_1$, we will make a series of \emph{attempts}. In the $i$th attempt we start by picking an arbitrary root vertex $v_i \in [n] \setminus X$ of a tree $\hat{T}_i$ and adding it to $X$, and we grow $\hat{T}_i$ by recursively exposing the neighbours in $U$ of some leaf $w$ in $\hat{T}_i$ at a minimal depth. If $w$ has at least two neighbours, we choose two arbitrarily and add them to $\hat{T}_i$ as children of $w$, and to the set $X$. If $w$ has less than two neighbours, the attempt is considered a \emph{failure} and we choose an arbitrary set of size $N_1 - |V(\hat{T}_i)|$ in $U$, add it to $X$ and begin the $(i+1)$th attempt. Otherwise, when $|V(\hat{T}_i)| = N_1$ then we consider the attempt a \emph{success} and begin the $(i+1)$th attempt.
 
Whenever we expose the neighbours of a vertex $w$ there are $|U| \geq \left(1-\frac{\delta}{4} \right)n$ many possible neighbours and so the probability that we do not find at least two neighbours is at most
\begin{align*}
		\mathbb{P}\left( \text{Bin}\left(\left(1-\frac{\delta}{4} \right)n , p \right) < 2 \right) &= (1-p)^{\left(1-\frac{\delta}{4} \right)n} + \left(1-\frac{\delta}{4} \right)np(1-p)^{\left(1-\frac{\delta}{4} \right)n-1}\\
		&\leq \left(1 - p + \left(1-\frac{\delta}{4} \right)\left(1+\delta\right)\right)\text{exp}\left(-(1+\delta)\left(1-\frac{\delta}{4} \right) + p\right)\\
		&\leq \left(2 + \frac{3 \delta}{4}\right) e^{-1-\frac{\delta}{2}}=: 1-\gamma < 1.
		\end{align*}
 It follows that the probability that an attempt is successful is at least $\gamma^{N_1}$. In particular, since these probabilities are independent for each attempt, if we make $k = \gamma^{-{N_1}}N_1$ attempts, then whp there will be at least one successful attempt, that is, some $\hat{T}_j$ such that $|V(\hat{T}_j)| = N_1$. Note that, at the end of this initial stage $|X| \leq \gamma^{-{N_1}}N_1^2 = o(n)$.
 
 Let us set $T_0 = \hat{T}_j$ and $S_0$ to be the set of leaves of $T_0$. Note that $|S_0| \geq N_2$. We will build a sequence of trees $T_i$ together with a specified set of leaves $S_i \subseteq V(T_i)$ such that for each $i \geq 0$
 \begin{itemize}
    \item $T_{i+1} \supseteq T_{i}$;
     \item $\Delta(T_i) \leq K$;
     \item $|S_{i+1}| \geq \left(1 + \frac{\delta}{2}\right)|S_{i}|$.
 \end{itemize}

To do so, let $|S_i| = s_i$ and let us enumerate $S_i = \{ v_1,v_2, \ldots v_{s_i} \}$. We sequentially expose the neighbours of $v_j$ in $U$, the vertices which are neither discarded, nor in the current tree, for each $1 \leq j \leq s_i$. We choose an arbitrary subset $N_j \subseteq N(v_j) \cap U$ of size $\eta_j:= \min \{ |N(v_j) \cap U|, K\}$, add these vertices as children of $v_j$ in the tree $S_{i+1}$, and add them to $X$.

Note that, as long as the tree we are building has not yet grown to size $\frac{\delta n}{8}$, $|X| \leq \frac{\delta n}{4}$ and so the random variables $(\eta_1, \eta_2, \ldots, \eta_{s_i})$ are stochastically dominated by an i.i.d sequence of random variables $(Y_1, Y_2 \ldots, Y_{s_i})$ where each $Y_j \sim Y$ with $Y= \min \left \{ \text{Bin}\left( \left(1-\frac{\delta}{4}\right)n, p\right), K \right\}$.

Then, by Lemma \ref{l:restricted}
\begin{align*}
    \mathbb{E}(Y)\geq \left(1-\frac{\delta}{4}\right)np-K2^{-K}=\left(1-\frac{\delta}{4}\right)(1+\delta)-4\delta^4\log \frac{1}{\delta}=1+\frac{3\delta}{4}-\frac{\delta^2}{4}-4\delta^4\log \frac{1}{\delta}\geq 1+\frac{\delta}{2},
\end{align*}
if $\delta$ is sufficiently small. Hence, $\mathbb{E}(|S_{i+1}|) \geq |S_i|\mathbb{E}(Y) \geq \left(1+\frac{\delta}{2}\right)|S_i|$ and so by Lemma \ref{l:hoeffding}
\[
\mathbb{P}\left( |S_{i+1}| \leq \left(1+\frac{\delta}{4}\right)|S_i| \right) \leq 2\exp\left(- \frac{ \delta^2|S_i|}{16K^2} \right).
\]

Therefore, the probability that, whilst $|X| \leq \frac{\delta}{4}n,$ there is some $i$ such that $|S_{i+1}| \leq \left(1+\frac{\delta}{4}\right)|S_i|$ is at most
\[
\sum_{i \colon |T_i| \leq \frac{\delta}{8}n} 2\exp\left(- \frac{ \delta^2|S_i|}{16K^2} \right) \leq \sum_{t=N_2}^{\infty}2\exp\left(- \frac{ \delta^2t}{16K^2} \right) = o(1).
\]

In particular, whp there is some $j$ such that $|T_j| \geq \frac{\delta}{8}$, and this tree satisfies the conclusion of the lemma.
\end{appendix}
\end{document}